\documentclass[reqno,12pt]{amsart}
\usepackage{amsmath,amsfonts, amstext, amsthm, amssymb,color, epsfig,indentfirst}
\usepackage{float}
\newtheorem{thm}{Theorem}[section]

\newtheorem{cor}{Corollary}[section]

\newtheorem{rem}{Remark}[section]

\newcommand{\Z}{\mathbb{Z}}

\newcommand{\ds}{\operatorname{ds}}

\def\ds{\displaystyle}

\begin{document}
\title[On Quadratic Forms and ECS]{Quadratic Forms, Exact Covering Systems, and Product Identities for Theta Functions}
\author{Zhu ~Cao}
\address{Department of Mathematics, Kennesaw State University, Marietta, GA 30060,  USA}
\email{zcao@kennesaw.edu}

\begin{abstract}
In this paper, we establish a connection between integral quadratic forms and exact covering systems (ECS) and present a structural framework for a class of product identities involving Ramanujan's theta functions. This approach yields infinitely many such identities. As applications, we provide a unified interpretation for twenty-two of Ramanujan's forty identities for the Rogers-Ramanujan functions. Many identities analogous to the forty identities can be naturally explained from this perspective. In addition, we discuss ternary quadratic forms and derive new identities involving products of three or more theta functions. We conclude by unifying several previous approaches and providing a summary.
\end{abstract}
\maketitle
\section {Introduction}
In his one-page communication \cite{RAM40NOTE}, Ramanujan stated, ``I have now found an algebraic relation between $G(q)$ and $H(q)$, viz.:
\begin{equation*}
G^{11}(q)H(q) - q^2G(q)H^{11}(q) = 1 + 11q G^6(q)H^6(q).
\end{equation*}
Another noteworthy formula is
\begin{align} \label{40ID4}
G(q^{11})H(q) - q^2G(q)H(q^{11}) = 1.
\end{align}
Each of these formulas is the simplest of a large class.''

Here, $G(q)$ and $H(q)$ denote the Rogers--Ramanujan functions, defined as
\[
G(q) := \sum_{n=0}^{\infty} \frac{q^{n^2}}{(q; q)_n}, \quad
H(q) := \sum_{n=0}^{\infty} \frac{q^{n(n+1)}}{(q; q)_n}
\]
for $\lvert q \rvert < 1$.

In this paper, we present a universal structure that accounts for a collection of such identities, including a natural proof of \eqref{40ID4}.

The remainder of this paper is organized as follows:

{Section 2} introduces the necessary notation and preliminary results that will be used in the proofs of the results presented in later sections.

{Section 3} presents, from the perspective of integral quadratic forms and exact covering systems (ECS) associated with integer matrices, proofs or sketches of proofs for 22 of Ramanujan’s 40 identities for the Rogers--Ramanujan functions. To prove a given identity, it is often necessary to identify a suitable quadratic form and expand it using two ECSs of \( \mathbb{Z}^n \), each corresponding to one side of the identity. This method, to our knowledge, has not appeared previously in the literature. The idea is illustrated in detail in the proof of \eqref{40ID4}, which is regarded as the most important proof in this paper.

It is important to emphasize that the purpose of this paper is not to provide complete proofs of all forty identities. Rather, we aim to showcase the power and elegance of this structural framework, while avoiding overly artificial manipulations whenever possible. The 22 identities presented here do not constitute the full set that can be proved using our approach. Some of the more intricate identities---whose proofs are omitted---require additional identities beyond the foundational results introduced in this paper.

In {Section 4}, since the approach introduced in this paper can be systematically employed to establish infinitely many new identities, we focus on ternary quadratic forms and present some new identities for the products of three theta functions as applications.  

{Section 5} continues with new identities and emphasizes that many identities analogous to Ramanujan’s forty identities can be naturally interpreted through our framework. We also provide a brief review of prior approaches, including Schröter’s formula and its generalizations, Watson’s method, Rogers’ method, Bressoud’s generalization of Rogers’ method, and Yesilyurt’s generalization of Bressoud’s method. Notably, the elementary techniques used in these previous works align with the matrix ECS structure developed in this paper.

\section {Notation and Preliminary Results}
We use the standard notation for $q$-products throughout this paper, defining
\begin{align*} 
(a)_0 &:= (a;q)_{0} := 1, \\
(a)_n &:= (a;q)_{n} := \prod_{k=0}^{n-1} (1 - aq^{k}), \quad n \geq 1,\\
(a)_\infty &:= (a;q)_{\infty} := \prod_{k=0}^{\infty} (1 - aq^{k}), \quad |q| < 1.
\end{align*}
Jacobi's triple product identity is given by \cite[p.~10]{spirit}
\begin{align} \label{JTI}
\sum_{n= -\infty}^{\infty} q^{n^{2}} z^{n} = (-qz;q^{2})_{\infty}
(-q/z;q^{2})_{\infty} (q^{2}; q^{2})_{\infty}, \qquad |q| < 1.
\end{align}
Ramanujan's general theta function is defined as
\begin{align} \label{theta}
f(a,b) := \sum^\infty_{n=-\infty} a^{n(n+1)/2} b^{n(n-1)/2}, \quad |ab| < 1.
\end{align}
By Jacobi's triple product identity, we have
\[
f(a,b) = (-a; ab)_\infty (-b; ab)_\infty (ab; ab)_\infty.
\]
It is easy to verify the following properties:
\begin{align}
f(a, b) &= f(b, a), \\
f(1, a) &= 2f(a, a^3), \\
\label{f-1a}
f(-1, a) &= 0,
\end{align}
and for any integer \( n \),
\begin{align} \label{aa}
f(a, b) = a^{n(n+1)/2} b^{n(n-1)/2} f(a(ab)^n, b(ab)^{-n}).
\end{align}
The following special cases of \( f(a, b) \) are frequently used in this paper:
\begin{align}
\varphi(q) &:= \sum_{n=-\infty}^{\infty} q^{n^2} = f(q,q) = (-q; q^2)^2_{\infty} (q^2; q^2)_{\infty}, \\
\psi(q) &:= \sum_{n=0}^{\infty} q^{n(n+1)/2} = f(q, q^3) = \frac{(q^2; q^2)_{\infty}}{(q; q^2)_{\infty}}, \\
f(-q) &:= f(-q, -q^2) = (q; q)_{\infty}.
\end{align}
We define \( \chi(q) := (-q; q^2)_{\infty} \), which is not a theta function but is introduced for notational convenience. In most of the forty identities for the Rogers–Ramanujan functions discussed in this paper, quotients of \( \chi \)-functions appear on the right-hand sides.

We will also use the following quotient form of the well-known quintuple product identity, as given in \cite{spirit}:
\begin{align} \label{QPI}
\frac{f(-x^2, -\lambda x)f(-\lambda x^3)}{f(-x, -\lambda x^2)} 
= f(-\lambda^2 x^3, -\lambda x^6) + x f(-\lambda, -\lambda^2 x^9).
\end{align}

In \cite{CAOECS}, the author proved a generalized form of the quintuple product identity by showing that it is a special case of the main theorem. For a detailed history and various proofs of the quintuple product identity, readers may refer to the survey by S. Cooper \cite{cooper}.

Let \( P(q) \) denote any power series in \( q \). Then the
\( t \)-dissection of \( P \) is given by
\begin{equation*}
P(q) := \sum_{k=0}^{t-1} q^k P_k(q^t),
\end{equation*}
where the \( P_k \) are the components of the dissection.

For a positive integer \( k \), corresponding to the exact covering system 
\[
\left\{ r \pmod{k} \right\}_{r=0}^{k-1},
\]
a theta function can be expressed as a linear combination of \( k \) theta functions:
\begin{align}
f(a,b) &= \sum_{n=-\infty}^{\infty} a^{n(n+1)/2} b^{n(n-1)/2} \notag \\
&= \sum_{r=0}^{k-1} \sum_{n=-\infty}^{\infty} a^{(kn+r)(kn+r+1)/2} b^{(kn+r)(kn+r-1)/2} \notag \\
\label{DISECT}
&= \sum_{r=0}^{k-1} a^{r(r+1)/2} b^{r(r-1)/2} f\big( a^{k(k+1)/2 + kr} b^{k(k-1)/2 + kr},\, a^{k(k-1)/2 - kr} b^{k(k+1)/2 - kr} \big).
\end{align}

Setting \( k = 2 \) in \eqref{DISECT}, we obtain
\begin{align} \label{k=2}
f(a,b) = f(a^3 b, a b^3) + a f(b/a, a^5 b^3).
\end{align}

The case \( k = 3 \) yields the 3-dissection of \( f(a, b) \):
\begin{align} \label{3dis}
f(a,b) = f(a^6 b^3, a^3 b^6) + a f(b^3, a^9 b^6) + b f(a^3, a^6 b^9),
\end{align}
which will be used in this paper.

The notion of covering system was first introduced by Paul Erd\"{o}s in the early 1930s. A covering system, also known as a complete residue system, is a collection of residue classes $a_i\pmod {n_i}$ with $1\leq i \leq k$ such that every integer belongs to least one of the residue classes. A covering system in which each integer is covered by exactly one congruence is called an exact covering system (ECS). In other words, an exact covering system is a partition of the integers into a finite set of arithmetic sequences.

In this paper, we focus on integer matrix exact covering systems (ECS) of $\Z^n$,  which are partitions of $\Z^n$ into a lattice and a finite set of its translates. For an $n\times n$ invertible integer matrix $B$, $B\Z^n \lhd \Z^n$ under the binary operation as vector addition in the abelian additive group $\Z^n$. If \( k = |\det B| \), then the quotient group \( \mathbb{Z}^n / B \mathbb{Z}^n \) consists of \( k \) cosets. Let \( \{R_1, \dots, R_k \}\) be a complete set of representatives for these cosets, then we have
\[
\mathbb{Z}^n = \bigcup_{i=1}^k (B \mathbb{Z}^n + R_i).
\]
In \cite{CAOHURECS}, the authors discussed a special type of matrix ECS with well-structured coset representatives, given by
\begin{equation} \label{ECS001}
\{ B\mathbb{Z}^n + i e_j \}, \quad i = 0, 1, \dots, k-1 \quad \left(\text{or } -\left\lfloor \frac{k}{2} \right\rfloor + 1 \leq i \leq \left\lfloor \frac{k}{2} \right\rfloor \right),
\end{equation}
where \( e_j = (0, \dots, 0, 1, 0, \dots, 0)^\mathrm{T} \) is the \( j \)-th column of the identity matrix \( I_n \). An ECS of the form \eqref{ECS001} is called a {simple ECS}, and the matrix \( B \) in \eqref{ECS001} is defined as a {simple covering matrix}. The authors provided the following criterion to identify a simple ECS.
\begin{thm} \label{SIMECS}
A nonsingular integer matrix \( B = (b_{ij})_{n \times n} \) is a simple covering matrix if and only if there exists \( 1 \le j \le n \) such that the entries of \( \mathbf{b}_j^* \), the \( j \)-th column of the adjugate matrix \( B^* \), are coprime. In this case, letting \( k = |\det B| \), the set \( \ds \bigcup_{i = 0}^{k-1} \{ B\mathbb{Z}^n + i e_j \} \) forms a simple ECS of \( \mathbb{Z}^n \).
\end{thm}
In \cite{CAOECS}, the author first introduced the application of integer matrix ECS to product identities involving Ramanujan's theta functions.

Let \( m_i \in \mathbb{Z}^{+} \) and suppose \( a_i b_i = q^{2m_i} \) for \( i = 1, 2, \dots, n \). We consider the product of \( n \) theta functions:
\begin{align*}
S &= \prod_{i=1}^n f(a_i, b_i) 
= \sum_{x_1, x_2, \dots, x_n = -\infty}^{\infty}
a_1^{\frac{x_1^2 + x_1}{2}} b_1^{\frac{x_1^2 - x_1}{2}} \cdots 
a_n^{\frac{x_n^2 + x_n}{2}} b_n^{\frac{x_n^2 - x_n}{2}} \\
&= \sum_{x_1, x_2, \dots, x_n = -\infty}^{\infty}
(a_1 b_1)^{\frac{x_1^2}{2}} \cdots (a_n b_n)^{\frac{x_n^2}{2}}
\left( \frac{a_1}{b_1} \right)^{\frac{x_1}{2}} \cdots 
= \sum_{x_1, x_2, \dots, x_n = -\infty}^{\infty}
q^{m_1 x_1^2 + \cdots + m_n x_n^2} \cdots
\end{align*}
The linear terms in \( x_i \) are omitted here.

Let \( X = (x_1, x_2, \dots, x_n)^\mathrm{T} \). Then the quadratic form appearing in $S$ can be written as
\[
m_1 x_1^2 + \cdots + m_n x_n^2 = X^\mathrm{T} A X,
\]
where
\[
A = \begin{pmatrix}
m_1 & 0 & 0 & \cdots & 0 \\
0 & m_2 & 0 & \cdots & 0 \\
\vdots & \vdots & \vdots & \ddots & \vdots \\
0 & 0 & 0 & \cdots & m_n
\end{pmatrix}
\]
is a positive definite diagonal matrix.

For \(X^\mathrm{T} A X\), we consider an ECS associated with a nonsingular integer matrix \( B \), and replace \( X \) by \( B Y + R_i \) for \( i = 1, \dots, k \). Then the quadratic term in $X^\mathrm {T}AX$ becomes $Y^\mathrm {T}B^\mathrm {T}ABY$. Motivated by this, we may find a matrix \( B \) such that \( B^\mathrm{T} A B \) is diagonal. In such cases, the transformed expression remains a product of theta functions. This leads to a general result on products of theta functions: Theorem 1.4 in \cite{CAOECS}, which shows that a product of \( n \) theta functions can be expressed as a linear combination of \( k \) terms, each of which is a product of \( n \) theta functions.

In particular, if \( A \) is a scalar matrix, then \( B^\mathrm{T} B \) is diagonal. In this context, we define such a matrix \( B \) as a generalized orthogonal matrix (with respect to \( A \)).

However, this approach is not limited to products of theta functions in which the quadratic terms appear as sums of squares; it can also be extended to general quadratic forms.

An \( n \)-ary quadratic form (QF) over \( \mathbb{Z} \) (also called an integral quadratic form) is a homogeneous polynomial of degree 2 in \( n \) variables with integer coefficients:
\[
f(x_1,\ldots,x_n) = \sum_{i, j=1}^{n} a_{ij} x_i x_j, \quad a_{ij} = a_{ji} \in \mathbb{Z}.
\]
This can be expressed in matrix form as
\[
f(x_1,\ldots,x_n) = X^{\mathrm{T}} A X,
\]
where \( X = (x_1, x_2, \ldots, x_n)^{\mathrm{T}} \), and \( A \) is a symmetric matrix, called the {associated matrix} of the quadratic form.

A quadratic form \( f(x_1,\ldots,x_n) \) is said to be {primitive} if \( \gcd(a_{ij}) = 1 \). It is called {positive definite} if \( f(X) > 0 \) for all \( X \neq \mathbf{0} \).

Two \( n \)-ary quadratic forms \( f = X^{\mathrm{T}} A X \) and \( g = X^{\mathrm{T}} C X \) are said to be {equivalent}, denoted \( f \sim g \), if there exists an invertible integral linear change of variables that transforms one into the other. That is, there exists an integer matrix \( B \) with \( \det B = \pm 1 \) such that
\[
B^{\mathrm{T}} A B = C.
\]

A quadratic form \(\ds  f(x_1,\ldots,x_n) = \sum_{i, j=1}^{n} a_{ij} x_i x_j \) is said to {represent} an integer \( m \) if there exist integers \( x_1, \ldots, x_n \) such that \( f(x_1,\ldots,x_n) = m \).

For an introduction to integral quadratic forms, the reader is referred to \cite{buell}.

For a polynomial in \( n \) variables of degree 2 with integer coefficients, written in the form \( f(X) \) for \( X \in \mathbb{Z}^n \), where the quadratic terms are represented by matrix \( A \). After applying the transformations \( X = B Y + R_i \), for \( i = 1, \ldots, k \), our goal is to eliminate the coefficients of all cross terms \( y_i y_j \) with \( i \neq j \) in the transformed expression. That is, we seek a matrix \( B \) such that $B^{\mathrm{T}} A B$
is diagonal. In this context, matrix \( A \) is not necessarily diagonal.

For the infinite sum
\[
\sum^{\infty}_{x_1, x_2=-\infty} q^{a x_1^2 + b x_1 x_2 + c x_2^2},
\]
where $a$, $b$, and $c$ are positive integers, the quadratic form $a x_1^2 + b x_1 x_2 + c x_2^2$ can be expressed as $X^{\mathrm{T}} A X$, where
\[
A = \begin{pmatrix}
a & b/2 \\
b/2 & c
\end{pmatrix}
\]
is not diagonal.

To diagonalize $A$, we choose an integral matrix $B$ such that $B^{\mathrm{T}} A B$ is diagonal. This requirement is equivalent to satisfying the condition
\[
2a b_{11} b_{12} + b (b_{11} b_{22} + b_{12} b_{21}) + 2c b_{21} b_{22} = 0,
\]
where $B = (b_{ij})$ is a $2 \times 2$ matrix.

In \cite{CAOECS}, we examined the following infinite sum:
\begin{align} \label{case111}
\sum^{\infty}_{x_1, x_2 = -\infty} q^{x_1^2 + x_1 x_2 + x_2^2}.
\end{align}

In this situation, as discussed in M.~D.~Hirschhorn's paper \cite{HIRSCH1987} on the proof of Jacobi's Four Square Theorem, it is a common practice to split the infinite sum into two parts:
\[
\sum_{x_1 + x_2 \text{ even}} q^{x_1^2 + x_1 x_2 + x_2^2}
+
\sum_{x_1 + x_2 \text{ odd}} q^{x_1^2 + x_1 x_2 + x_2^2}.
\]
In the first sum, we set
\[
y_1 = \frac{x_1 - x_2}{2}, \quad y_2 = \frac{x_1 + x_2}{2},
\]
and in the second sum, we set
\[
y_1 = \frac{x_1 - x_2 - 1}{2}, \quad y_2 = \frac{x_1 + x_2 - 1}{2}.
\]
This approach is equivalent to changing variables from $(x_1, x_2)$ to $(y_1, y_2)$ using transformations corresponding to an ECS of $\mathbb{Z}^2$ induced by the matrix
\[
B = \begin{pmatrix}
1 & 1 \\
-1 & 1
\end{pmatrix}.
\]
Here, the linear transformations from $X$ to $Y$ are given by
\[
\left\{
\begin{pmatrix}
x_1 \\
x_2
\end{pmatrix}
=
\begin{pmatrix}
1 & 1 \\
-1 & 1
\end{pmatrix}
\begin{pmatrix}
y_1 \\
y_2
\end{pmatrix}
+
\begin{pmatrix}
i \\
0
\end{pmatrix}
: i = 0, 1
\right\}.
\]
It is easy to see that the lattice $B \mathbb{Z}^2$ covers half of the points in $\mathbb{Z}^2$, while the shifted lattice $B \mathbb{Z}^2 + \begin{pmatrix} 1 \\ 0 \end{pmatrix}$ covers the remaining half.

As such, we can write \eqref{case111} as the sum of two parts
\begin{align*}
&\sum_{x_1, x_2 = -\infty}^{\infty} q^{x_1^2 + x_1 x_2 + x_2^2}\\
&= \sum_{y_1, y_2 = -\infty}^{\infty} q^{(y_1 + y_2)^2 + (y_1 + y_2)(y_2 - y_1) + (y_2 - y_1)^2} + \sum_{y_1, y_2 = -\infty}^{\infty} q^{(y_1 + y_2 + 1)^2 + (y_1 + y_2 + 1)(y_2 - y_1) + (y_2 - y_1)^2} \\
&= \sum_{y_1, y_2 = -\infty}^{\infty} q^{y_1^2 + 3 y_2^2}
+ \sum_{y_1, y_2 = -\infty}^{\infty} q^{y_1^2 + y_1 + 3 y_2^2 + 3 y_2}\\
&= \varphi(q)\varphi(q^3) + 4q \psi(q^2) \psi(q^6),
\end{align*}
which was first discovered by Ramanujan and discussed in \cite{BB1992} and \cite{spirit}.
\begin{rem}
This matrix $B$ has determinant $2$. Since an integral matrix $B$ with determinant $1$ cannot diagonalize a non-diagonal matrix $A$ (in the sense that $B^{\mathrm{T}} A B$ becomes diagonal), this particular $B$ induces the simplest nontrivial matrix ECS of $\mathbb{Z}^2$ considered in this paper.
\end{rem}

Matrix ECS has direct applications to infinite sums. Let $\ds \bigcup_{i = 1}^k \{B\mathbb{Z}^n + R_i\}$ be a matrix ECS, then
\begin{equation} \label{ECS1}
f(X) = \sum_{i = 1}^k f(BX + R_i)
\end{equation}
for any convergent infinite sum $f(X)$ over $\mathbb{Z}^n$.

We can extend the definition of equivalence of quadratic forms from the perspective of matrix ECS by removing the restriction that $\det B = \pm 1$. By replacing $X$ with $\ds \bigcup_{i=1}^k (BX + R_i)$ in $X^{\mathrm{T}} A X$, we obtain
\begin{align}
X^{\mathrm{T}} A X \sim \bigcup_{i=1}^k \left( X^{\mathrm{T}} B^{\mathrm{T}} A B X + R_i^{\mathrm{T}} (A^{\mathrm{T}} + A) B X + R_i^{\mathrm{T}} A R_i \right).
\end{align}
After this transformation, the resulting expression is no longer homogeneous in $X$. We refer to such expressions as extended quadratic forms, which include both linear terms in $X$ and constant terms.

In this paper, we focus on infinite sums of the form
\begin{align} \label{QFGF}
\sum_{X \in \mathbb{Z}^n, \Delta \in \mathbb{Z}_2^n} (-1)^{\Delta^{\mathrm{T}} X} q^{X^{\mathrm{T}} A X + BX + C},
\end{align}
where $\mathbb{Z}_2^n$ is the set of all vectors in $\mathbb{Z}^n$ whose entries are restricted to $0$ or $1$, $B$ is a $1 \times n$ integral matrix, and $C$ is an integer.

Since most product identities for theta functions involve products of two theta functions, we primarily consider quadratic forms in two variables.

A {binary quadratic form} (BQF) is a function of the form
\[
f(x, y) = ax^2 + 2bxy + cy^2, \quad (a, b, c \in \mathbb{Z}),
\]
and is denoted by $(a, 2b, c)$. We use $2b$ instead of $b$, following Gauss’s convention in his \emph{Disquisitiones Arithmeticae} \cite{gauss}. The discriminant of $f(x, y)$ is defined by
\[
D := 4b^2 - 4ac.
\]
Gauss proved that for each value of $D$, there exist only finitely many equivalence classes of binary quadratic forms with discriminant $D$. The number of such classes is called the \emph{class number} of discriminant $D$.

For a quadratic form over $\mathbb{Z}$ of the form
\[
f(x_1, x_2) = a x_1^2 + 2b x_1 x_2 + c x_2^2,
\]
where $a$, $b$, and $c$ are integers, the associated matrix is
\[
A = \begin{pmatrix}
a & b \\
b & c
\end{pmatrix}.
\]
The determinant of $A$ is $ac - b^2 = -D/4$. In this paper, we are only interested in the case where $a > 0$ and $D < 0$, which implies that $A$ is positive definite.

The general form of infinite sums involving extended BQFs that appear in this paper is
\begin{align} \label{SUMBQF}
\sum_{x_1, x_2 = -\infty}^{\infty} (-1)^{\delta_1 x_1 + \delta_2 x_2} q^{a x_1^2 + 2b x_1 x_2 + c x_2^2 + d x_1 + e x_2 + f},
\end{align}
where $\delta_1, \delta_2 \in \mathbb{Z}_2$ and $a, b, c, d, e, f \in \mathbb{Z}$.

Throughout the paper, we assume that $a$, $b$, and $c$ are nonnegative integers satisfying $2b \leq a \leq c$. For convenience, we use the term {determinant} instead of {discriminant}.
\section {Forty Identities for the Rogers-Ramanujan Functions}
\subsection{Introduction of the Forty Identities}\

\noindent
It can be shown that Rogers-Ramanujan functions satisfy the famous Rogers-Ramanujan identities
\begin{align} \label{RRID}
G(q)=\frac{1}{(q; q^5)_\infty(q^4; q^5)_\infty}\quad \hbox{and} \quad H(q)=\frac{1}{(q^2; q^5)_\infty(q^3; q^5)_\infty}.
\end{align}
Rogers-Ramanujan identities were first proven by L.~J.~Rogers in \cite{rogers1894} and later rediscovered and proved by Ramanujan in \cite{RAM40IDS}. Ramanujan established a list of forty identities involving the Rogers-Ramanujan functions $G(q)$ and $H(q)$, but provided no proofs. The forty identities was discovered in 1975 by J. Birch in G. N. Watson's handwritten notes, archived in the Oxford University Library, and was later included in \cite{RAMLOST}. Watson remarked: ``The beauty of these formulae seems to me to be comparable with that of the Rogers-Ramanujan identities. So far as I know, nobody else has discovered any formulae which approach them even remotely...''.

A lot of effort has been devoted to studying the forty identities and additional identities for $G(q)$ and $H(q)$, as well as analogous identities for functions of the Rogers-Ramanujan type. Various mathematical tools have been developed to tackle them. In 1921, H. B. C. Darling \cite{DAR1921} was the first to publish a proof of one identity on the list. Soon after, Rogers \cite{ROGERS1921} proved ten identities. In 1933, Watson \cite{watson} proved eight more, two of which with overlapped with those proved by Rogers. In his Ph.D. dissertation \cite{BREPHD}, D. ~Bressoud established proofs for another fifteen identities from the list. He later published a paper \cite{BRE1977} containing proofs of general identities from \cite{BREPHD}, developed using the method originally employed by Rogers. After their contributions, nine identities from the list remained unproven. In 1989, Biagioli \cite{BIA1989} proved eight of them using modular forms. In B.~C.~Berndt \emph{et al.} \cite{BER40ID}, thirty-five of the forty identities were proven using a variety of techniques. In 2009, H.~Yesilyurt \cite{YESJNT} presented further generalizations of Rogers' method and used them to prove three more identities. Finally, in 2012, Yesilyurt \cite{YESJMAA} proved three identities using the extension of Roger's method he developed in \cite{YESJNT}, including the only remaining open identity. 

For the most comprehensive reference on the forty identities, readers are encouraged to consult \cite{BER40ID} and Chapter~8 of \cite[pp.~217--335]{lost3}, the latter of which includes Yesilyurt’s contributions made after the publication of \cite{BER40ID}.

As Birch commented, ``all the functions involved in the identities are essentially theta functions $\cdots$''. From Jacobi's triple product identity and the definition of theta functions, both $G(q)$ and $H(q)$ can be expressed as the quotient of two theta functions:
\[G(q) = \frac{f(-q^2, -q^3)}{f(-q)}\,\, \text{and}\,\, H(q) = \frac{f(-q, -q^4)}{f(-q)}.\]
This representation allows us to rewrite all the identities proven in this paper to be rewritten  in terms of product of theta functions.

In \cite{BER40ID}, since some identities are interconnected, the authors treated the list as 35 entries rather than 40. For convenience, as it is the most comprehensive paper on this topic, we follow the entry numbering used therein. For example, we label the first part of Entry 3.16 in \cite{BER40ID} as Identity 16-1 in this paper.

In this section, we provide (or outline the key information necessary for) the proofs of twenty-two identities among Ramanujan's forty identities. We must acknowledge that we have greatly benefited from the works \cite{BER40ID}, \cite{BREPHD}, \cite{YESJNT}, and \cite{YESJMAA}. Some of the proofs presented here are direct adaptations from these references, reformulated in the language of quadratic forms and ECS.

We classify identities addressed in this section into three types.

Type I: Identity 2, 5, 7, 8, 11, and 12 are special cases of the main theorem for the product of two theta functions in \cite{CAOECS}, Theorem 3.1 introduced later in this section. They correspond to the case where the quadratic form represented by $A$ contains no cross terms. We expand one side of the identity, represented as the product of two theta functions, using ECS to obtain the other side, which is a linear combination of products of two theta functions.

Type II: Identities 4, 6, 9, 10, 19-1, and 28-1 correspond to the case where the matrix $A$ for the associated quadratic form is not diagonal. In this case, we identify an appropriate quadratic form and expand it using two ECSs of $\Z^2$ to generate each side of the identity. 

Type III: Identities 13, 14, 16-1, 17, 18, 19-2, 22, 23, 24, and 30. These identities are more intricate. Their proofs either rely on auxiliary results such as the quintuple product identity or require combining various ECSs to establish the identities.

These classifications are not mutually exclusive. Identities 5, 11, and
12 can be treated as both Type I and Type II. In the sections that follow, we provide
detailed proofs for representative cases from each type.

\subsection{Type I Identities}\

First we list all the identities of Type I. \\
\textbf{Identity 2.}
\begin{align} \label{ID402}
G(q)G(q^4)+qH(q)H(q^4)&= \chi^2(q)= \frac{\varphi(q)}{f(-q^2)}.
\end{align}
\textbf{Identity 5.}
\begin{align} \label{ID405}
G(q^{16})H(q)-q^3G(q)H(q^{16})&=\frac{f^2(-q^3)}{f(-q)f(-q^9)}.
\end{align}
\textbf{Identity 7.}
\begin{align} \label{ID407}
G(q^2)G(q^3)+qH(q^2)H(q^3)&=\frac{\chi(-q^3)}{\chi(-q)}.
\end{align}
\textbf{Identity 8.}
\begin{align} \label{ID408}
G(q^6)H(q)-qG(q)H(q^6)&=\frac{\chi(-q)}{\chi(-q^3)}.
\end{align}
\textbf{Identity 11.}
\begin{align} \label{ID4011}
G(q^8)H(q^3)-qG(q^3)H(q^8)&=\frac{\chi(-q)\chi(-q^4)}{\chi(-q^3)\chi(-q^{12})}.
\end{align}
\textbf{Identity 12.}
\begin{align} \label{ID4012}
G(q)G(q^{24})+q^5H(q)H(q^{24})&=\frac{\chi(-q^3)\chi(-q^{12})}{\chi(-q)\chi(-q^4)}.
\end{align}

By Theorem \ref{SIMECS}, we have the following theorem for the product of two theta functions. It is a slightly more generalized version of Theorem 2.1 in \cite{CAOECS}.

\begin{thm} \label{ECSCASE2}
Let $ab=q^{m_1}$ and $cd=q^{m_2}$, where $m_1, m_2 \in \Z^{+}$. Let $B=(b_{ij})$ be a $2\times 2$ non-singular integer matrix such that \[m_1b_{11}b_{12}+m_2b_{21}b_{22}=0\] and \[\gcd(b_{21}, b_{22})=1.\] Let $k=|\det B|$. Then the following equality holds.
\begin{align}
&f(a, b)f(c,d) \notag \\
=&\sum^{k-1}_{i=0}{a}^{\frac{i^2+i}{2}}{b}^{\frac{i^2-i}{2}}\notag\\ 
&f(a^{\frac{b_{11}^2+b_{11}}{2}+b_{11}i}b^{\frac{b_{11}^2-b_{11}}{2}+b_{11}i}
c^{\frac{b_{21}^2+b_{21}}{2}}d^{\frac{b_{21}^2-b_{21}}{2}},
a^{\frac{b_{11}^2-b_{11}}{2}-b_{11}i}b^{\frac{b_{11}^2+b_{11}}{2}-b_{11}i}
c^{\frac{b_{21}^2-b_{21}}{2}}d^{\frac{b_{21}^2+b_{21}}{2}}) \notag \\
\label{2genthmc} 
&f(a^{\frac{b_{12}^2+b_{12}}{2}+b_{12}i}b^{\frac{b_{12}^2-b_{12}}{2}+b_{12}i}
c^{\frac{b_{22}^2+b_{22}}{2}}d^{\frac{b_{22}^2-b_{22}}{2}},
a^{\frac{b_{12}^2-b_{12}}{2}-b_{12}i}b^{\frac{b_{12}^2+b_{12}}{2}-b_{12}i}
c^{\frac{b_{22}^2-b_{22}}{2}}d^{\frac{b_{22}^2+b_{22}}{2}}).
\end{align}
\end{thm}
For the matrix $B$, if $\gcd(b_{11}, b_{12})=1$, we have another analogous version of the theorem.

Let \[c_1=m_1{b_{11}}^2+m_2{b_{21}}^2, \,c_2=m_1{b_{12}}^2+m_2{b_{22}}^2,\] it is easy to see that for the theta functions on the right-hand
side of \eqref{2genthmc}, we have
\begin{align}
&a^{\frac{b_{11}^2+b_{11}}{2}+b_{11}i}b^{\frac{b_{11}^2-b_{11}}{2}+b_{11}i}
c^{\frac{b_{21}^2+b_{21}}{2}}d^{\frac{b_{21}^2-b_{21}}{2}}\cdot
a^{\frac{b_{11}^2-b_{11}}{2}-b_{11}i}b^{\frac{b_{11}^2+b_{11}}{2}-b_{11}i}
c^{\frac{b_{21}^2-b_{21}}{2}}d^{\frac{b_{21}^2+b_{21}}{2}}\notag \\
=&(ab)^{{b_{11}}^2}(cd)^{{b_{21}}^2}=q^{m_1{b_{11}}^2+m_2{b_{21}}^2} = q^{c_1},\notag \\
&a^{\frac{b_{12}^2+b_{12}}{2}+b_{12}i}b^{\frac{b_{12}^2-b_{12}}{2}+b_{12}i}
c^{\frac{b_{22}^2+b_{22}}{2}}d^{\frac{b_{22}^2-b_{22}}{2}}\cdot
a^{\frac{b_{12}^2-b_{12}}{2}-b_{12}i}b^{\frac{b_{12}^2+b_{12}}{2}-b_{12}i}
c^{\frac{b_{22}^2-b_{22}}{2}}d^{\frac{b_{22}^2+b_{22}}{2}})\notag \\
\label{NECB}
=&(ab)^{{b_{12}}^2}(cd)^{{b_{22}}^2}=q^{m_1{b_{12}}^2+m_2{b_{22}}^2} = q^{c_2}.\\
\notag
\end{align}
We will use \eqref{NECB}  to guide the search for matrices $B$ associated with product identities of theta functions. 

Of the forty identities, the six listed in this subsection can be proved as special cases of Theorem \ref{ECSCASE2}.

We first give a proof of Identity 7.
\begin{proof}
We rewrite \eqref{ID407} in the form of product of theta functions
\begin{align} \label{id1}
2f(-q^4, -q^6)f(-q^6, -q^9)+2qf(-q^2, -q^8)f(-q^3, -q^{12}) =f(1, q)f(-q^3, -q^3).
\end{align}
Our goal is to find an appropriate ECS to transform the right-hand side of \eqref{id1} to the left-hand side.  Note that $1\cdot q=q$ and $-q^3\cdot -q^3=q^6$, , so we have $m_1=1$ and $m_2=6$, as in Theorem \ref{ECSCASE2}. Also observe that $-q^4\cdot-q^6=-q^{2}\cdot -q^8=q^{10}, -q^6\cdot-q^9=-q^3\cdot-q^{12}=q^{15}$. Thus, by \eqref{NECB}, we obtain $c_1 = 10, c_2 = 15.$ 

We now seek an integral matrix $B$ satisfying
\begin{equation*}
B^T \left(\begin{array}{cc} 1 & 0 \\ 0 & 6 \end{array}\right) B
=\left(\begin{array}{cc} 10 & 0 \\ 0 & 15 \end{array}\right).
\end{equation*}
Equating the corresponding entries of these two matrices yields the following system of equations
\begin{equation}\label{304eqh1}
\left\{
\begin{array}{rcrcrcrcr}
b^2_{11} &+& 6b^2_{21}&=&10,\\
b^2_{12} &+& 6b^2_{22}&=&15,\\
b_{11}b_{12} &+& 6b_{21}b_{22}&=&0.\\
\end{array}
\right.
\end{equation}

 Since both $10$ and $15$ can be represented by the quadratic form $x^2 + 6y^2$, the first two equations have integer solutions. These solutions can be used to determine the entries of the matrix $B$, with the signs remaining initially undetermined. The third equation, serving as an orthogonality condition, is then used to resolve the signs of the entries of $B$. 

Since each row or column can be multiplied by $-1$ and still satisfy the system of equations, although the solution matrix $B$ is not unique, we can choose $B$ such that its determinant and as many entries as possible are positive. In this case, as well as in most others discussed in this paper, there is only one ECS determined by $B$.

We find that
\[
B = \begin{pmatrix}
2 & 3 \\
-1 & 1
\end{pmatrix}
\]
is a solution to \eqref{304eqh1}. By Theorem~\ref{ECSCASE2}, we can choose an ECS as
\[
\begin{pmatrix}
2 & 3 \\
-1 & 1
\end{pmatrix}
\mathbb{Z}^2 +
\begin{pmatrix} i \\ 0 \end{pmatrix}, \quad i = -2, \ldots, 2.
\]

Letting $a = 1$, $b = q$, and $c = d = -q^3$ in Theorem~\ref{ECSCASE2}, we obtain
\begin{align*}
f(1, q)f(-q^3, -q^3)
&= \sum_{i = -2}^{2} q^{\frac{i^2 - i}{2}} f(-q^{4 + 2i}, -q^{6 - 2i}) f(-q^{6 + 3i}, -q^{9 - 3i}) \\
&= q^3 f(-1, -q^{10}) f(-1, -q^{15}) + 2f(-q^4, -q^6)f(-q^6, -q^9) \\
&\quad + 2q f(-q^2, -q^8) f(-q^3, -q^{12}).
\end{align*}
By \eqref{f-1a}, we have $f(-1, -q^{10}) = f(-1, -q^{15}) = 0$, so \eqref{id1} follows after simplification.
\end{proof}
For the proofs of the other five identities of this type, we provide only the key information below.

{Identity 2} \eqref{ID402} is equivalent to
\[
2f(-q^2, -q^3)f(-q^8, -q^{12}) + 2qf(-q, -q^4)f(-q^4, -q^{16}) = f(1, q)f(-q^2, -q^2).
\]
The associated ECS is
\[
\begin{pmatrix}
1 & 4 \\
-1 & 1
\end{pmatrix}
\mathbb{Z}^2 +
\begin{pmatrix} i \\ 0 \end{pmatrix}, \quad i = -2, \ldots, 2.
\]
{Identity 5} \eqref{ID405} is equivalent to
\[
4f(-q, -q^4)f(-q^{32}, -q^{48}) - 4q^3f(-q^2, -q^3)f(-q^{16}, -q^{64}) = f(1, -q)f(1, -q^4).
\]
The associated ECS is
\[
\begin{pmatrix}
1 & 8 \\
-1 & 2
\end{pmatrix}
\mathbb{Z}^2 +
\begin{pmatrix} i \\ 0 \end{pmatrix}, \quad i = -4, \ldots, 5.
\]
{Identity 8} \eqref{ID408} is equivalent to
\[
2f(-q, -q^4)f(-q^{12}, -q^{18}) - 2qf(-q^2, -q^3)f(-q^6, -q^{24}) = f(1, q)f(-q, -q).
\]
The associated ECS is
\[
\begin{pmatrix}
1 & 3 \\
-1 & 2
\end{pmatrix}
\mathbb{Z}^2 +
\begin{pmatrix} i \\ 0 \end{pmatrix}, \quad i = -2, \ldots, 2.
\]
{Identity 11} \eqref{ID4011} is equivalent to
\[
4f(-q^3, -q^{12})f(-q^{16}, -q^{24}) - 4qf(-q^6, -q^9)f(-q^8, -q^{32}) = f(1, -q)f(1, -q^6).
\]
The associated ECS is
\[
\begin{pmatrix}
3 & 4 \\
-1 & 2
\end{pmatrix}
\mathbb{Z}^2 +
\begin{pmatrix} i \\ 0 \end{pmatrix}, \quad i = -4, \ldots, 5.
\]
{Identity 12} \eqref{ID4012} is equivalent to
\[
4f(-q^2, -q^3)f(-q^{48}, -q^{72}) + 4q^5f(-q, -q^4)f(-q^{24}, -q^{96}) = f(1, -q^2)f(1, -q^3).
\]
The associated ECS is
\[
\begin{pmatrix}
1 & 6 \\
-1 & 4
\end{pmatrix}
\mathbb{Z}^2 +
\begin{pmatrix} i \\ 0 \end{pmatrix}, \quad i = -4, \ldots, 5.
\]
\subsection{Type II Identities}\

The list of type II identities is given as follows.

\noindent
\textbf{Identity 4.}
\begin{align} \label{ID404}
H(q)G(q^{11})-q^2G(q)H(q^{11})=1.
,\end{align}
\textbf{Identity 6.} 
\begin{align} \label{ID406}
G(q)G(q^9)+q^2H(q)H(q^9)&=\frac{f^2(-q^3)}{f(-q)f(-q^9)}.
\end{align}
\textbf{Identity 9.} 
\begin{align} \label{ID409}
G(q^7)H(q^2)-qG(q^2)H(q^7)&=\frac{\chi(-q)}{\chi(-q^7)}.
\end{align}
\textbf{Identity 10.} 
\begin{align} \label{ID4010}
G(q)G(q^{14})+q^3H(q^2)H(q^{14})&=\frac{\chi(-q^7)}{\chi(-q)}.
\end{align}
\textbf{Identity 19-1.} 
\begin{align} \label{ID40191}
\frac{G(q)g(q^{39})+q^8H(q)H(q^{39})}{G(q^{13})H(q^3)-q^2G(q^3)H(q^{13})}
=\frac{f(-q^3)f(-q^{13})}{f(-q)f(-q^{39})}.
\end{align}
\textbf{Identity 28-1} 
\begin{align} \label{ID40281}
\frac{G(q^{17})H(q^2)-q^3G(q^2)H(q^{17})}{G,(q)G(q^{34})+q^7H(q)H(q^{34})}
&=\frac{\chi(-q)}{\chi(-q^{17})}.
\end{align}

Identity 4 is \eqref{40ID4}, one of the two most remarkable identities among the forty identities, according to Ramanujan.  
Here we provide a detailed proof of Identity 4.

\begin{proof}
In terms of theta functions, \eqref{ID404} can be rewritten as
\begin{align*}
f(-q)f(-q^{11}) = f(-q, -q^4)f(-q^{22}, -q^{33}) - q^2 f(-q^2, -q^3)f(-q^{11}, -q^{44}),
\end{align*}
or equivalently,
\begin{align} \label{40mb1}
f(-q, -q^2)f(-q^{11}, -q^{22}) = f(-q, -q^4)f(-q^{22}, -q^{33}) - q^2 f(-q^2, -q^3)f(-q^{11}, -q^{44}).
\end{align}

From the definition of Ramanujan's theta function, \eqref{40mb1} is equivalent to
\begin{align} \label{40mb2}
&\sum_{y_1, y_2 = -\infty}^{\infty} (-1)^{y_1 + y_2} q^{3y_1^2 + y_1 + 33y_2^2 + 11y_2}\\
=& \sum_{y_1, y_2 = -\infty}^{\infty} (-1)^{y_1 + y_2} q^{5y_1^2 + 3y_1 + 55y_2^2 + 11y_2} - q^2 \sum_{y_1, y_2 = -\infty}^{\infty} (-1)^{y_1 + y_2} q^{5y_1^2 + y_1 + 55y_2^2 + 33y_2}.
\end{align}

Since the left-hand side of \eqref{40mb2} is a single product of two theta functions, and the right-hand side consists of two such products, we may expect that the left-hand side can be expanded into the right-hand side via an ECS associated with an integral matrix $B$ of determinant $2$.

However, in this case,
\[
B^{\mathrm{T}} 
\begin{pmatrix}
3 & 0 \\
0 & 33
\end{pmatrix}
B = 
\begin{pmatrix}
5 & 0 \\
0 & 55
\end{pmatrix}
\] has no integral matrix solution since $\det B=5/3$.

We find that a solution matrix to the above equation is
\[
B = \frac{1}{3}
\begin{pmatrix}
2 & 11 \\
-1 & 2
\end{pmatrix}.
\]
This corresponds to the rational matrix ECS (RECS) introduced in \cite{CAOHURECS}. However, RECS is less convenient than ECS for our purposes, as it requires prior knowledge of the identity and lacks the intuitive structural clarity offered by ECS.

We observe that the quadratic form $(3, 0, 33)$ appearing on the left-hand side of \eqref{40mb2} has determinant $3 \cdot 33 = 99 = 3^2 \cdot 11$. On the other hand, the determinant of the quadratic form $(5, 0, 55)$ on the right-hand side is $5 \cdot 55 = 275 = 5^2 \cdot 11$. This observation motivates us to begin with an extended quadratic form having determinant $11$, and to find two matrices of determinants $3$ and $5$, respectively, such that this quadratic form can be expanded into the left- and right-hand sides of the identity, according to the corresponding matrix ECS.

There are only three reduced primitive forms of quadratic forms with determinant $11$: $(1,0,11)$, $(3,2,4)$, and $(3,-2,4)$. Since $(3,2,4)$ and $(3,-2,4)$ are equivalent here, we only need to consider
\begin{align} \label{ID4S}
S = \sum_{x_1,\, x_2 = -\infty}^{\infty} (-1)^{\delta_1 x_1 + \delta_2 x_2} q^{3x_1^2 + 2x_1 x_2 + 4x_2^2 + dx_1 + ex_2},
\end{align}
with $d$ and $e$ still to be determined.

In order to recover the left-hand side of \eqref{40mb2}, we must find an integral matrix $B$ satisfying
\[
B^{\mathrm{T}} 
\begin{pmatrix}
3 & 1 \\
1 & 4
\end{pmatrix}
B = 
\begin{pmatrix}
3 & 0 \\
0 & 33
\end{pmatrix}.
\]
This leads to the following system of equations:
\begin{equation} \label{ID4eqh2}
\left\{
\begin{aligned}
3b_{11}^2 + 2b_{11}b_{21} + 4b_{21}^2 &= 3, \\
3b_{12}^2 + 2b_{12}b_{22} + 4b_{22}^2 &= 33, \\
3b_{11}b_{12} + b_{11}b_{22} + b_{12}b_{21} + 4b_{21}b_{22} &= 0.
\end{aligned}
\right.
\end{equation}

From the first equation of \eqref{ID4eqh2}, applying the quadratic formula yields  
\[
b_{11} = \frac{-b_{21} \pm \sqrt{9 - 11b_{21}^2}}{3}.
\]
We require \( 9 - 11b_{21}^2 \) to be a perfect square. Clearly, the only integer solution is \( b_{21} = 0 \), which gives \( b_{11} = \pm 1 \). Without loss of generality, we assume \( b_{11} = 1 \).

Similarly, solving the second equation of \eqref{ID4eqh2}, we find \( b_{12} = -1 \) and \( b_{22} = 3 \). It can be verified that this set of solutions satisfies the third equation, known as the ``orthogonality relation.''

Let
\[
B_1 = 
\begin{pmatrix}
1 & -1 \\
0 & 3
\end{pmatrix}.
\]
It is evident that the first row of \( B_1 \) contains two odd entries, allowing us to choose \( \delta_1 = 1 \) and \( \delta_2 = 0 \) in \eqref{ID4S}. According to the ECS induced by \( B_1 \), we apply the transformation
\[
X = B_1 Y + 
\begin{pmatrix}
0 \\
i
\end{pmatrix}, \quad (i = -1, 0, 1),
\]
to rewrite \( S \) as a linear combination of products of theta functions.

The contribution of \( i = 0 \) to \( S \) in this case is
\[
\sum_{y_1,\, y_2 = -\infty}^{\infty} (-1)^{y_1 + y_2} q^{3y_1^2 + d y_1 + 33 y_2^2 + (3e - d)y_2}.
\]
By choosing \( d = 1 \) and \( e = 4 \), we obtain \( f(-q^2, -q^4)f(-q^{22}, -q^{44}) \).

Now we have determined the exact form of \( S \):
\begin{align} \label{ID4SUM}
S = \sum_{x_1,\, x_2 = -\infty}^{\infty} (-1)^{x_1} q^{3x_1^2 + 2x_1 x_2 + 4x_2^2 + x_1 + 4x_2}.
\end{align}

We can verify that the contribution for \( i = -1 \) is the same as that for \( i = 0 \). For \( i = 1 \), the contribution to \( S \) is
\[
q^8 \sum_{y_1,\, y_2 = -\infty}^{\infty} (-1)^{y_1 + y_2} q^{3y_1^2 + 3y_1 + 33y_2^2 + 33y_2}.
\]
This equals \( q^8 f(-1, -q^6)f(-1, -q^{66}) \), which is zero by \eqref{f-1a}.

Summing the contributions for each value of \( i \), we obtain
\begin{align} \label{eq1}
S = 2f(-q^2, -q^4)f(-q^{11}, -q^{22}) = 2f(-q^2)f(-q^{22}).
\end{align}

Similarly, we consider the matrix equation
\[
B^{\mathrm{T}} 
\begin{pmatrix}
3 & 1 \\
1 & 4
\end{pmatrix}
B =
\begin{pmatrix}
5 & 0 \\
0 & 55
\end{pmatrix},
\]
which corresponds to the system of equations:
\begin{equation*}
\left\{
\begin{aligned}
3b_{11}^2 + 2b_{11}b_{21} + 4b_{21}^2 &= 5, \\
3b_{12}^2 + 2b_{12}b_{22} + 4b_{22}^2 &= 55, \\
3b_{11}b_{12} + b_{11}b_{22} + b_{12}b_{21} + 4b_{21}b_{22} &= 0.
\end{aligned}
\right.
\end{equation*}

A solution is given by
\[
B_2 = 
\begin{pmatrix}
1 & 3 \\
-1 & 2
\end{pmatrix},
\]
and we omit the details here.

We expand \eqref{ID4SUM} using the ECS corresponding to \( B_2 \), letting
\[
X = B_2 Y + 
\begin{pmatrix}
i \\
0
\end{pmatrix}, \quad i = -2, \ldots, 2,
\]
in \( S \), which yields
\begin{align} \label{eq2}
S = 2f(-q^2, -q^8)f(-q^{44}, -q^{66}) + 2q^4 f(-q^4, -q^6)f(-q^{22}, -q^{88}).
\end{align}
By \eqref{eq1} and \eqref{eq2}, we obtain \eqref{40mb1} after replacing \( q^2 \) with \( q \), which is an equivalent form of Identity 4.
\end{proof}

\begin{rem}
The quadratic form \( (3,2,4) \) is unique in that it represents all of 3, 33, 5, and 55. A remarkable symmetry is observed in the proof of Identity 4: here, \( \det B_1 = 3 \) and \( \det B_2 = 5 \), indicating that the left-hand side of the identity consists of the sum of three terms, while the right-hand side consists of five. The identity thus takes the form
\[
2A + 0 = 2B + 2C + 0,
\]
which simplifies to
\[
A = B + C.
\]

This symmetry likely explains why Ramanujan described it as “the simplest of a large class.” Based on the author’s calculations, \( (3,2,4) \) is the only quadratic form that can be diagonalized by two matrices with determinants 3 and 5 simultaneously. Therefore, Identity 4 appears to be the only identity that follows the pattern \( 2A + 0 = 2B + 2C + 0 \).

Although Ramanujan’s original proof of Identity 4 remains unknown, the author believes that the proof presented here is the simplest and most natural possible. Moreover, infinitely many analogous identities can be constructed using the same approach.

From the proof, it is evident that, unlike the case where the matrix \( A \) representing the quadratic form is diagonal, formulating a general theorem for the product of \( n \) theta functions—similar to Theorem 1.4 in \cite{CAOECS}—in the non-diagonal case is highly impractical. In fact, such a general theorem is unnecessary.

Rather than assuming the identities in advance, we may begin with positive definite extended quadratic forms of small determinants and identify all favorable cases as the determinant increases. This process is closely related to the representation of integers by quadratic forms. The coefficients of the linear terms \( x_i \) are chosen to produce the desired theta functions appearing in the identity.
\end{rem}

The details of the proofs for the following identities are omitted.

{Identity 6} \eqref{ID406} is equivalent to
\[
f(-q^{2}, -q^{3})f(-q^{18}, -q^{27}) + q^2 f(-q, -q^{4})f(-q^{9}, -q^{36}) = f^2(-q^3).
\]
The corresponding infinite sum involving an extended quadratic form is
\[
S = \sum_{x_1,\, x_2 = -\infty}^{\infty} (-1)^{x_1} q^{2x_1^2 + 2x_1 x_2 + 5x_2^2 + x_1 + 2x_2},
\]
with associated matrices
\[
B_1 = 
\begin{pmatrix}
1 & 1 \\
-2 & 1
\end{pmatrix}, \quad
B_2 = 
\begin{pmatrix}
0 & -5 \\
1 & 1
\end{pmatrix}.
\]
{Identity 9} \eqref{ID409} is equivalent to
\[
f(-q^{2}, -q^{8})f(-q^{14}, -q^{21}) - qf(-q^4, -q^6)f(-q^7, -q^{28}) = f(-q)f(-q^{14}).
\]
The associated sum is
\[
S = \sum_{x_1,\, x_2 = -\infty}^{\infty} (-1)^{x_1} q^{3x_1^2 + 2x_1 x_2 + 5x_2^2 + x_1 + 4x_2},
\]
with
\[
B_1 = 
\begin{pmatrix}
1 & -1 \\
0 & 3
\end{pmatrix}, \quad
B_2 = 
\begin{pmatrix}
1 & -3 \\
1 & 2
\end{pmatrix}.
\]
{Identity 10} \eqref{ID4010} is equivalent to
\[
f(-q^{2}, -q^{3})f(-q^{28}, -q^{42}) + q^3 f(-q, -q^4)f(-q^{14}, -q^{56}) = f(-q^2)f(-q^7).
\]
The corresponding sum is
\[
S = \sum_{x_1,\, x_2 = -\infty}^{\infty} (-1)^{x_2} q^{3x_1^2 + 2x_1 x_2 + 5x_2^2 + 3x_1 + x_2},
\]
with matrices
\[
B_1 = 
\begin{pmatrix}
1 & 2 \\
-1 & 1
\end{pmatrix}, \quad
B_2 = 
\begin{pmatrix}
0 & -5 \\
1 & 1
\end{pmatrix}.
\]
{Identity 19-1} \eqref{ID40191} is equivalent to
\begin{align*}
&f(-q^{2}, -q^{3})f(-q^{78}, -q^{117}) + q^8 f(-q, -q^{4})f(-q^{39}, -q^{156}) \\
&= f(-q^3, -q^{12})f(-q^{26}, -q^{39}) - q^2 f(-q^3, -q^{12})f(-q^{13}, -q^{52}).
\end{align*}
The associated sum is
\[
S = \sum_{x_1,\, x_2 = -\infty}^{\infty} (-1)^{x_1} q^{5x_1^2 + 2x_1 x_2 + 8x_2^2 + x_1 + 8x_2},
\]
with matrices
\[
B_1 = 
\begin{pmatrix}
1 & -1 \\
0 & 5
\end{pmatrix}, \quad
B_2 = 
\begin{pmatrix}
1 & -3 \\
1 & 2
\end{pmatrix}.
\]
{Identity 28-1} \eqref{ID40281} is equivalent to
\begin{align*}
&f(-q^2, -q^8)f(-q^{34}, -q^{51}) - q^3 f(-q^4, -q^6)f(-q^{17}, -q^{68}) \\
&= f(-q^2, -q^3)f(-q^{68}, -q^{102}) + q^7 f(-q, -q^4)f(-q^{34}, -q^{136}).
\end{align*}
The corresponding sum is
\[
S = \sum_{x_1,\, x_2 = -\infty}^{\infty} (-1)^{x_1} q^{5x_1^2 + 2x_1 x_2 + 7x_2^2 + x_1 + 7x_2},
\]
with
\[
B_1 = 
\begin{pmatrix}
1 & 3 \\
-1 & 2
\end{pmatrix}, \quad
B_2 = 
\begin{pmatrix}
1 & -1 \\
0 & 5
\end{pmatrix}.
\]

\begin{rem}
Identities 5, 11, and 12 can be classified as both Type I and Type II. For a positive integer \( \alpha \), we have
\[
f(1, -q^\alpha) = 2f(-q^\alpha, -q^{3\alpha}) = 2\psi(-q^\alpha),
\]
so the classification depends on which form appears on the right-hand side of the identity. We include them in Type I, as its structure is simpler than that of Type II.
\end{rem}

If treated as Type II, {Identity 5} \eqref{ID405} is equivalent to
\[
f(-q, -q^4)f(-q^{32}, -q^{48}) - q^3 f(-q^2, -q^3)f(-q^{16}, -q^{64}) = f(-q, -q^3)f(-q^4, -q^{12}).
\]
Associated sum is
\[
S = \sum_{x_1,\, x_2 = -\infty}^{\infty} (-1)^{x_1 + x_2} q^{4x_1^2 + 4x_1 x_2 + 5x_2^2 + 2x_1 - 3x_2},
\]
with matrices
\[
B_1 = 
\begin{pmatrix}
1 & -1 \\
0 & 2
\end{pmatrix}, \quad
B_2 = 
\begin{pmatrix}
0 & -5 \\
1 & 2
\end{pmatrix}.
\]
{Identity 11} \eqref{ID4011} is equivalent to
\[
f(-q^3, -q^{12})f(-q^{16}, -q^{24}) - q f(-q^6, -q^9)f(-q^8, -q^{32}) = f(-q, -q^3)f(-q^6, -q^{18}).
\]
Associated sum is
\[
S = \sum_{x_1,\, x_2 = -\infty}^{\infty} (-1)^{x_1} q^{4x_1^2 + 4x_1 x_2 + 7x_2^2 + 2x_1 + 7x_2},
\]
with matrices
\[
B_1 = 
\begin{pmatrix}
1 & -1 \\
0 & 2
\end{pmatrix}, \quad
B_2 = 
\begin{pmatrix}
1 & -3 \\
1 & 2
\end{pmatrix} \quad \text{or} \quad
\begin{pmatrix}
2 & 1 \\
-1 & 2
\end{pmatrix}.
\]
This case is unusual since two different matrix ECSs correspond to \( B_2 \). Both expansions yield the same infinite sum.

{Identity 12} \eqref{ID4012} is equivalent to
\[
f(-q^2, -q^3)f(-q^{48}, -q^{72}) + q^5 f(-q, -q^4)f(-q^{24}, -q^{96}) = f(-q^2, -q^6)f(-q^3, -q^9).
\]
Associated sum is
\[
S = \sum_{x_1,\, x_2 = -\infty}^{\infty} (-1)^{x_1} q^{5x_1^2 + 2x_1 x_2 + 5x_2^2 + x_1 + 5x_2},
\]
with
\[
B_1 = 
\begin{pmatrix}
1 & 1 \\
-1 & 1
\end{pmatrix}, \quad
B_2 = 
\begin{pmatrix}
1 & -1 \\
0 & 5
\end{pmatrix}.
\]
\subsection{Type III Identities}\

The list of ten Type III identities is given below.

\noindent
\textbf{Identity 13.} 
\begin{align} \label{ID4013}
G(q^9)H(q^4) - qG(q^4)H(q^9) &= \chi(-q)\chi(-q^6)\chi(-q^3)\chi(-q^{18}). \
\end{align}
\textbf{Identity 14.} 
\begin{align} \label{ID4014}
G(q^{36})H(q) - q^7 G(q)H(q^{36}) &= \chi(-q^6)\chi(-q^9)\chi(-q^2)\chi(-q^3). 
\end{align}
\textbf{Identity 16-1.}
\begin{align} \label{ID40161}
G(q^2)G(q^{13})+q^3H(q^2)H(q^{13})&=G(q^{26})H(q)-q^5G(q)H(q^{26}).
\end{align}
\textbf{Identity 17.} 
\begin{align} \label{ID4017}
&G(q)G(q^{19})+q^4H(q)H(q^{19}) \notag \\
&=\frac{1}{4\sqrt{q}}\chi^2(q^{1/2}) \chi^2(q^{19/2})-\frac{1}{4\sqrt{q}}
\chi^2(q^{-1/2})\chi^2(-q^{19/2})-\frac{q^2}{\chi^2(-q^2)\chi^2(-q^{19})}.
\end{align}
\textbf{Identity 18.} 
\begin{equation} \label{ID4018}
G(q^{31})H(q) -  q^6G(q)H(q^{31}) =  \frac{1}{2q}\chi(q)\chi(q^{31})-\frac{1}{2q}\chi(-q)\chi(-q^{31}) + \frac{q^3}{\chi(-q^2)\chi(-q^{62})}.
\end{equation}
\textbf{Identity 19-2.} 
\begin{align} \label{ID40192}
\{G(q)G(q^{39})+q^8H(q)H(q^{39})\}f(-q)f(-q^{39})
=&\frac{1}{2q}\Big\{\varphi(-q^3) \varphi(-q^{13}) - \varphi(-q) \varphi(-q^{39})\Big\}.
\end{align}
\textbf{Identity 22.}
\begin{align} \label{ID4022}
G(-q)G(-q^4)+qH(-q)H(-q^{4})&=\chi(q^2).
\end{align}

\subsection*{Identity 23.}
\begin{equation} \label{ID4023}
G(-q^2)G(-q^3) + qH(-q^2)H(-q^3) = 
\frac{\chi(q)\chi(q^6)}{\chi(q^2)\chi(q^3)}.
\end{equation}

\subsection*{Identity 24.}
\begin{equation} \label{ID4024}
G(-q^6)H(-q) - qH(-q^6)G(-q) = 
\frac{\chi(q^2)\chi(q^3)}{\chi(q)\chi(q^6)}.
\end{equation}
\textbf{Identity 30.} 
\begin{equation} \label{ID4030}
\frac{G(q^{19})H(q^4)-q^3G(q^4)H(q^{19})}{G(q^{76})H(-q)+q^{15}G(-q)H(q^{76})} = \frac{\chi(-q^2)}{\chi(-q^{38})}.
\end{equation}

We first provide a detailed proof of Identity 16-1 \eqref{ID40161}.
\begin{proof}
In terms of theta functions, this identity is equivalent to
\begin{align} 
&\frac{f(-q^{4}, -q^{6})f(-q^{26}, -q^{39}) + q^3 f(-q^{2}, -q^{8})f(-q^{13}, -q^{52})}{(q^2; q^2)_\infty (q^{13}; q^{13})_\infty} \notag \\ \label{ID161eq}
=\, &\frac{f(-q, -q^{4})f(-q^{52}, -q^{78}) - q^4 f(-q^{2}, -q^{3})f(-q^{26}, -q^{104})}{(q; q)_\infty (q^{26}; q^{26})_\infty}.
\end{align}

Since \( 10 \cdot 65 = 5 \cdot 130 = 650 = 26 \cdot 5^2 \), we may first consider quadratic forms with determinant 26.

However, we cannot find a single quadratic form that generates the theta functions appearing on both the left- and right-hand sides of \eqref{ID161eq} via ECS. The reason lies in the list of reduced primitive quadratic forms of determinant 26:
\[
(1,0,26),\ (2,0,13),\ (3,2,9),\ (3,-2,9),\ (5,4,6),\ (5,-4,6).
\]
Among these, only two need to be considered here: \( (3,2,9) \) and \( (5,4,6) \).

For \( (3,2,9) \), the associated matrix is
\[
A = \begin{pmatrix}
3 & 1 \\
1 & 9
\end{pmatrix}.
\]
Using this matrix, we find that the equation
\[
B^{\mathrm{T}} A B = \begin{pmatrix}
10 & 0 \\
0 & 65
\end{pmatrix}
\]
has a solution. However, the equation
\[
B^{\mathrm{T}} A B = \begin{pmatrix}
5 & 0 \\
0 & 130
\end{pmatrix}
\]
has no solution. Therefore, this quadratic form cannot generate both sides of the identity.

In contrast, consider the quadratic form associated with
\[
A = \begin{pmatrix}
5 & 2 \\
2 & 6
\end{pmatrix}.
\]
In this case, the equation
\[
B^{\mathrm{T}} A B = \begin{pmatrix}
10 & 0 \\
0 & 65
\end{pmatrix}
\]
has no solution, while
\[
B^{\mathrm{T}} A B = \begin{pmatrix}
5 & 0 \\
0 & 130
\end{pmatrix}
\]
does. Hence, we again cannot find a quadratic form of determinant 26 that accommodates both sides of the identity.

We now consider variants of the theta functions appearing in \eqref{ID161eq} and examine quadratic forms with determinant 39, which is \( \frac{3}{2} \cdot 26 \). Among all the binary quadratic forms of this determinant, only one primitive form meets our requirements: \( (5,2,8) \).

For the sum
\[
S = \sum_{x_1, x_2 \in \mathbb{Z}} (-1)^{\delta_1 x_1 + \delta_2 x_2} q^{5x_1^2 + 2x_1 x_2 + 8x_2^2 + d x_1 + e x_2},
\]
the associated matrix is
\[
A = \begin{pmatrix}
5 & 1 \\
1 & 8
\end{pmatrix}.
\]
It is easy to verify that both
\[
B^{\mathrm{T}} A B =
\begin{pmatrix}
15 & 0 \\
0 & 65
\end{pmatrix}
\quad \text{and} \quad
B^{\mathrm{T}} A B =
\begin{pmatrix}
5 & 0 \\
0 & 195
\end{pmatrix}
\]
have solutions. The corresponding matrices are
\[
B_1 = 
\begin{pmatrix}
1 & -1 \\
0 & 5
\end{pmatrix}, \qquad
B_2 = 
\begin{pmatrix}
1 & -3 \\
1 & 2
\end{pmatrix}.
\]
From the structure of \( B_1 \) and \( B_2 \), we choose \( \delta_1 = 1 \) and \( \delta_2 = 0 \).

Beginning with the ECS associated with \( B_1 \), the contribution of the case \( i = 0 \) to \( S \) is
\[
\sum_{y_1, y_2 \in \mathbb{Z}} (-1)^{y_1 + y_2} q^{5y_1^2 + 195y_2^2 + d y_1 + (-d + 5e) y_2}.
\]
To match the required theta functions, we set \( d = 1 \) and \( e = -5 \).

Adding up the contributions from \( i = -2, \ldots, 2 \) to \( S \), we obtain
\begin{align} \label{ID161S1}
S =\,& f(-q^4, -q^6) \big\{ f(-q^{13 \cdot 13}, -q^{13 \cdot 17}) + q^{13} f(-q^{13 \cdot 7}, -q^{13 \cdot 23}) \big\} \notag \\
& + q^3 f(-q^2, -q^8) \big\{ f(-q^{13 \cdot 11}, -q^{13 \cdot 19}) + q^{13 \cdot 3} f(-q^{13 \cdot 1}, -q^{13 \cdot 29}) \big\}.
\end{align}

By the quintuple product identity \eqref{QPI}, we simplify \eqref{ID161S1} to
\begin{align} \label{ID161S2}
S =\,& f(-q^4, -q^6) \cdot \frac{f(-q^{26}, -q^{104}) f(-q^{130})}{f(-q^{13}, -q^{117})} \notag \\
& + q^3 f(-q^2, -q^8) \cdot \frac{f(-q^{52}, -q^{78}) f(-q^{130})}{f(-q^{39}, -q^{91})}.
\end{align}

We record two identities from \cite{YESJMAA}, which can be verified easily.
\begin{equation} \label{JTPCORO}
\frac{f(-q^2, -q^8)}{f(-q, -q^9)} = \frac{f(-q^2, -q^3)}{\chi(-q) f(-q^{10})}, \qquad
\frac{f(-q^4, -q^6)}{f(-q^3, -q^7)} = \frac{f(-q, -q^4)}{\chi(-q) f(-q^{10})}.
\end{equation}
These identities establish connections between \( G(q^2) \) and \( H(q) \), as well as \( H(q^2) \) and \( G(q) \).

Applying \eqref{JTPCORO} to \eqref{ID161S2}, we obtain
\begin{align}
S =\,& \frac{f(-q^4, -q^6) f(-q^{26}, -q^{39}) + q^3 f(-q^2, -q^8) f(-q^{13}, -q^{52})}{\chi(-q^{13})} \notag \\
\label{ID3161}
=\,& \frac{ \left\{ G(q^2) G(q^{13}) + q^3 H(q^2) H(q^{13}) \right\} f(-q^2) f(-q^{13}) }{\chi(-q^{13})}.
\end{align}

On the other hand, expanding \( S \) using the ECS associated with \( B_2 \), we have
\begin{align}
S =\,& \frac{f(-q, -q^4) f(-q^{52}, -q^{78}) - q^5 f(-q^2, -q^3) f(-q^{26}, -q^{104})}{\chi(-q)} \notag \\
\label{ID3162}
=\,& \frac{ \left( H(q) G(q^{26}) - q^5 G(q) H(q^{26}) \right) f(-q) f(-q^{26}) }{\chi(-q)}.
\end{align}

We complete the proof of Identity 16-1 by equating the two expressions \eqref{ID3161} and \eqref{ID3162} and performing appropriate cancellations.
\end{proof}

\begin{rem}
Identity 16-1 is essentially of Type II, but it is more complicated. To prove it, we first “zoom out” the identity and then “zoom in” using the quintuple product identity and \eqref{JTPCORO}. The proofs of Identity 13 and Identity 30, presented later in this section, follow the same strategy.
\end{rem}

Identity 13 can be rewritten as
\begin{align} \label{3131}
&f(-q^4, -q^{16})f(-q^{18}, -q^{27}) - q f(-q^8, -q^{12})f(-q^9, -q^{36}) \notag \\
=&\,(q^4; q^4)_\infty (q^9; q^9)_\infty \cdot \frac{\chi(-q)\chi(-q^6)}{\chi(-q^3)\chi(-q^{18})}.
\end{align}
Similarly, Identity 14 can be rewritten as
\begin{align} \label{3141}
&f(-q, -q^4)f(-q^{72}, -q^{108}) - q^7 f(-q^2, -q^3)f(-q^{36}, -q^{144}) \notag \\
=&\,(q; q)_\infty (q^{36}; q^{36})_\infty \cdot \frac{\chi(-q^6)\chi(-q^9)}{\chi(-q^2)\chi(-q^3)}.
\end{align}
Since \( 20 \cdot 45 = 900 = 36 \cdot 5^2 \), we may begin by considering quadratic forms with determinant 36. We expand the sum
\[
S = \sum_{x_1, x_2 \in \mathbb{Z}} (-1)^{x_1 + x_2} q^{5x_1^2 + 4x_1 x_2 + 8x_2^2 + 3x_1 - 6x_2}
\]
using the ECS associated with the matrices
\[
B_1 = \begin{pmatrix}
2 & 1 \\
-1 & 2
\end{pmatrix}, \quad
B_2 = \begin{pmatrix}
1 & -2 \\
0 & 5
\end{pmatrix}, \quad
B_3 = \begin{pmatrix}
1 & 2 \\
-1 & 1
\end{pmatrix}.
\]

These expansions yield
\begin{align}
&f(-q^4, -q^{16})f(-q^{18}, -q^{27}) - q f(-q^8, -q^{12})f(-q^9, -q^{36}) \notag \\
\label{1314}
=&\, f(-q, -q^4)f(-q^{72}, -q^{108}) - q^7 f(-q^2, -q^3)f(-q^{36}, -q^{144}) \\
\label{13141}
=&\, \psi(q^9)\varphi(-q^{18}) - q f(q^3, q^6)f(-q^6, -q^{30}).
\end{align}

Equation \eqref{1314} shows that the left-hand sides of \eqref{3131} and \eqref{3141} are identical, as previously proved in \cite{BER40ID}. It is also shown in \cite{BER40ID} that the right-hand sides of \eqref{3131} and \eqref{3141} are equal. Hence, Identities 13 and 14 are equivalent.

We now offer an alternative proof of Identity 13. Following the approach used in Identity 16-1, we consider a quadratic form with determinant 54 instead of 36.

\begin{proof}
Let
\[
S = \sum_{x_1, x_2 \in \mathbb{Z}} (-1)^{x_1} q^{7x_1^2 + 6x_1 x_2 + 9x_2^2 + 6x_1},
\]
whose associated quadratic form is represented by the matrix
\[
A = \begin{pmatrix}
7 & 3 \\
3 & 9
\end{pmatrix}.
\]
Since
\[
B^{\mathrm{T}} A B = 
\begin{pmatrix}
10 & 0 \\
0 & 135
\end{pmatrix}
\]
has the solution
\[
B_1 = \begin{pmatrix}
1 & 3 \\
-1 & 2
\end{pmatrix},
\]
we expand \( S \) using the ECS associated with \( B_1 \) to obtain
\begin{align*}
S =\,& f(-q^4, -q^{16}) \left\{ f(-q^{117}, -q^{153}) + q^9 f(-q^{63}, -q^{207}) \right\} \notag \\
& - q f(-q^8, -q^{12}) \left\{ f(-q^{99}, -q^{171}) + q^{27} f(-q^9, -q^{261}) \right\}.
\end{align*}
Applying the quintuple product identity \eqref{QPI}, we rewrite \( S \) as
\begin{align} \label{313id1}
S =\,& f(-q^4, -q^{16}) \cdot \frac{f(-q^{18}, -q^{72}) f(-q^{90})}{f(-q^9, -q^{81})} \notag \\
& - q f(-q^8, -q^{12}) \cdot \frac{f(-q^{36}, -q^{54}) f(-q^{90})}{f(-q^{27}, -q^{63})} \notag \\
=\,& \frac{f(-q^4, -q^{16}) f(-q^{18}, -q^{27}) - q f(-q^8, -q^{12}) f(-q^9, -q^{36})}{\chi(-q^9)} \notag \\
=\,& \frac{ \left( H(q^4) G(q^9) - q G(q^4) H(q^9) \right) f(-q^4) f(-q^9) }{\chi(-q^9)},
\end{align}
where the second equality follows from \eqref{JTPCORO}.

Since
\[
B^{\mathrm{T}} A B = 
\begin{pmatrix}
9 & 0 \\
0 & 54
\end{pmatrix}
\]
has a solution
\[
B_2 = \begin{pmatrix}
0 & -3 \\
1 & 1
\end{pmatrix},
\]
we may also expand \( S \) using the ECS associated with \( B_2 \) to obtain
\begin{align} \label{313id2}
S = f(-q^{36}) \left\{ f(q^9, q^9) - q f(q^3, q^{15}) \right\} = f(-q^{36}) \cdot \frac{f(-q^2, -q^4) f(-q^6)}{f(q, q^5)},
\end{align}
where the last equality uses the quintuple product identity \eqref{QPI} with \( x = -q \), \( \lambda = -q^3 \).
\end{proof}

By \eqref{313id1} and \eqref{313id2}, we can prove Identity 13 after simplifications.

For Identity 30 \eqref{ID4030}, its first proof is given by Biagioli using modular forms. It is among the three identities proved by Yesilyurt in \cite{YESJNT} using elementary method. The proof here is Yesilyurt's proof in the language of quadratic form and ECS.
\begin{proof}
We consider the quadratic form (10, 8, 13) of determinant 114, instead of a quadratic form of determinant 76. For
\[
S = \sum_{x_1, x_2 \in \Z} (-1)^{x_1 + x_2} q^{10x_1^2 + 8x_1x_2 + 13x_2^2 + 6x_1 + 10x_2},
\]
we derive the following
\begin{align}
\notag S =& f(-q^4, -q^{16})\Big\{f(-q^{19\cdot 13}, -q^{19\cdot 17}) + q^{19}f(-q^{19\cdot 7}, -q^{19\cdot 23})\Big\}\label{ID30A} \\
&-q^3f(-q^8, -q^{12})\Big\{f(-q^{19\cdot 11}, -q^{19\cdot 19})+q^{57}(-q^{19\cdot1}, -q^{19\cdot29})\Big\}\\
=& f(-q^4, -q^{16}) \frac{f(-q^{19\cdot2}, -q^{19\cdot8})f(-q^{190})}{f(-q^{19\cdot1}, -q^{19\cdot9})}-q^3f(-q^8, -q^{12}) \frac{f(-q^{19\cdot4}, -q^{19\cdot6})f(-q^{190})}{f(-q^{19\cdot3}, -q^{19\cdot7})} \label{ID30B} \\
= &\frac{f(-q^4, -q^{16})f(-q^{19\cdot 2}, -q^{19\cdot3})- q^3f(-q^8, -q^{12})f(-q^{19\cdot1}, -q^{19\cdot4})}{\chi(-q^{19})} \label{ID30C} \\
\notag =& \Big\{G(q^{19})H(q^4)-q^3G(q^4)H(q^{19})\Big\}(q^4;q^4)_{\infty}
(q^{38};q^{38})_{\infty}.
\end{align}
In the above steps, \eqref{ID30A} is obtained using the matrix ECS transformation
\[\begin{pmatrix}
x_1 \\
x_2
\end{pmatrix}
=
\begin{pmatrix}
1 & -2 \\
0 & 5
\end{pmatrix}
\begin{pmatrix}
y_1 \\
y_2
\end{pmatrix}
+
\begin{pmatrix}
0 \\
i
\end{pmatrix}, \quad i = -2, -1, 0, 1, 2.
\]
Equation \eqref{ID30B} is derived from the quintuple product identity, and \eqref{ID30C} follows from identities \eqref{JTPCORO}.

We now apply a second matrix ECS transformation
\[
\begin{pmatrix}
x_1 \\
x_2
\end{pmatrix}
=
\begin{pmatrix}
1 & 3 \\
-1 & 2
\end{pmatrix}
\begin{pmatrix}
y_1 \\
y_2
\end{pmatrix}
+
\begin{pmatrix}
i \\
0
\end{pmatrix}, \quad i = -2, -1, 0, 1, 2,
\]
to obtain the alternative expression
\[
S = G(q^{76}) H(-q) + q^{15} G(-q) H(q^{76}) \cdot (q^2; q^2)_\infty (q^{76}; q^{76})_\infty.
\]
The identity is proved by equating the two expressions for \( S \) and simplifying both sides.
\end{proof}
Identities 17, 18, and 19-2 all require the use of 2-dimensional ECS combined with 1-dimensional ECS. The proof of Identity 18 presented here follows essentially the same reasoning as the proof by D. Bressoud in his Ph.D. dissertation \cite{BREPHD}. In that proof, Bressoud employed restricted quadratic forms that involve congruence conditions. However, the ECS-based approach presented here reveals the structure more clearly.

\begin{proof}
We can rewrite Identity 18 \eqref{ID4018} as
\begin{align} \label{ID18eq}
&2q f(-q, -q^4)f(-q^{62}, -q^{93}) - 2q^7 f(-q^2, -q^3)f(-q^{31}, -q^{124}) \notag \\
=&\, \varphi(-q^2)\varphi(-q^{62}) - \varphi(-q)\varphi(-q^{31}) + 2q^4 \psi(-q)\psi(-q^{31}).
\end{align}

From the theta functions appearing in \eqref{ID18eq}, we observe that the diagonal matrices representing the quadratic forms are
\[
\begin{pmatrix} 5 & 0 \\ 0 & 155 \end{pmatrix}, \quad
\begin{pmatrix} 4 & 0 \\ 0 & 124 \end{pmatrix}, \quad
\begin{pmatrix} 2 & 0 \\ 0 & 62 \end{pmatrix}.
\]
This suggests beginning with quadratic forms of determinant 31.

Consider
\[
A(q) := \sum_{x_1, x_2 \in \mathbb{Z}} (-1)^{x_1} q^{4x_1^2 + 2x_1 x_2 + 8x_2^2}.
\]
Expanding this using the ECS corresponding to the matrix
\[
\begin{pmatrix}
1 & -1 \\
0 & 4
\end{pmatrix},
\]
we obtain
\begin{align} \label{aq0}
A(q) = \varphi(-q^4)\varphi(-q^{124}) + 2q^8 \psi(-q^2)\psi(-q^{62}).
\end{align}

Similarly, consider
\[
B(q) := \sum_{x_1, x_2 \in \mathbb{Z}} (-1)^{x_1} q^{2x_1^2 + 2x_1 x_2 + 16x_2^2}.
\]
Expanding this using the ECS corresponding to
\[
\begin{pmatrix}
1 & -1 \\
0 & 2
\end{pmatrix},
\]
we find that 
\begin{align} \label{bq0}
B(q) = \varphi(-q^2)\varphi(-q^{62}).
\end{align}

Thus, we can express the right-hand side of \eqref{ID18eq} as \( A(q) - B(q) \), provided we replace \( q \) by \( q^2 \).

By considering the parity of \( x_1 \) in \( A(q) \) and \( B(q) \), which corresponds to a 1-dimensional ECS, we obtain
\begin{align} \label{aq1}
A(q) &= \sum_{x_1, x_2 \in \mathbb{Z}} q^{16x_1^2 + 4x_1 x_2 + 8x_2^2}
- \sum_{x_1, x_2 \in \mathbb{Z}} q^{16x_1^2 + 4x_1 x_2 + 8x_2^2 + 16x_1 + 2x_2 + 4},
\end{align}
and
\begin{align} \label{bq}
B(q) &= \sum_{x_1, x_2 \in \mathbb{Z}} q^{8x_1^2 + 4x_1 x_2 + 16x_2^2}
- \sum_{x_1, x_2 \in \mathbb{Z}} q^{16x_1^2 + 4x_1 x_2 + 8x_2^2 + 8x_1 + 2x_2 + 2}.
\end{align}

Next, we interchange \( x_1 \) and \( x_2 \) in \eqref{aq1}, yielding
\begin{align} \label{aq}
A(q) = \sum_{x_1, x_2 \in \mathbb{Z}} q^{8x_1^2 + 4x_1 x_2 + 16x_2^2}
- \sum_{x_1, x_2 \in \mathbb{Z}} q^{8x_1^2 + 4x_1 x_2 + 16x_2^2 + 2x_1 + 16x_2 + 4}.
\end{align}

Taking the difference of \eqref{aq} and \eqref{bq}, we have
\begin{align} \label{aq-bq}
A(q) - B(q) &= \sum_{x_1, x_2 \in \mathbb{Z}} q^{16x_1^2 + 4x_1 x_2 + 8x_2^2 + 8x_1 + 2x_2 + 2} \notag \\
&\quad - \sum_{x_1, x_2 \in \mathbb{Z}} q^{8x_1^2 + 4x_1 x_2 + 16x_2^2 + 2x_1 + 16x_2 + 4}.
\end{align}

Define
\[
C(q) := \sum_{x_1, x_2 \in \mathbb{Z}} (-1)^{x_1 + x_2} q^{5x_1^2 + 4x_1 x_2 + 7x_2^2 + 3x_1 - 5x_2 + 2}.
\]
All three quadratic forms involved—\( (4,2,8) \), \( (2,2,16) \), and \( (5,4,7) \)—have determinant 31.

To simplify \( C(q) \), we consider the parity of \( x_1 + x_2 \) and apply the change of variables
\[
\begin{cases}
x_1 \mapsto x_1 - x_2, \quad x_2 \mapsto -x_1 - x_2 & \text{if } x_1 + x_2 \text{ is even}, \\
x_1 \mapsto x_1 - x_2 - 1, \quad x_2 \mapsto -x_1 - x_2 & \text{if } x_1 + x_2 \text{ is odd}.
\end{cases}
\]
This transformation corresponds to the ECS associated with the simplest covering matrix \( B \) discussed in Remark 1.1. However, in this case, the matrix transforms the quadratic form \( (5,4,7) \) into \( (8,4,16) \), which still contains cross terms.

Using this transformation, we split \( C(q) \) into two parts
\begin{align} \label{cq}
C(q) &= \sum_{x_1, x_2 \in \mathbb{Z}} q^{x_1^2 + 4x_1 x_2 + 16x_2^2 + 8x_1 + 2x_2 + 2} \notag \\
&\quad - \sum_{x_1, x_2 \in \mathbb{Z}} q^{8x_1^2 + 4x_1 x_2 + 16x_2^2 + 16x_1 + 2x_2 + 4}.
\end{align}

From the expressions for \( A(q) \), \( B(q) \), and \( C(q) \), we find that
\[
C(q) = A(q) - B(q).
\]

Finally, we apply the ECS transformation corresponding to the matrix
\[
\begin{pmatrix}
1 & -2 \\
0 & 5
\end{pmatrix}
\]
to expand \( C(q) \) as
\[
C(q) = 2q^2 f(-q^2, -q^8)f(-q^{124}, -q^{186}) - 2q^{14} f(-q^4, -q^6)f(-q^{62}, -q^{248}),
\]
which matches the left-hand side of \eqref{ID18eq} if we replace $q^2$ by $q$ in \( C(q) \).
\end{proof}

Next, we provide the proof of Identity 17 \eqref{ID4017}.

\begin{proof}
After simplification, it suffices to prove that
\begin{align} \label{id17}
&4q f(-q^4, -q^6)f(-q^{76}, -q^{114}) + 4q^9 f(-q^2, -q^8)f(-q^{38}, -q^{152}) \notag \\
=&\, \varphi(q)\varphi(q^{19}) - \varphi(-q)\varphi(-q^{19}) - 4q^5 \psi(q^2)\psi(q^{38}).
\end{align}

Let
\[
S = \sum_{x_1, x_2 \in \mathbb{Z}} (-1)^{x_2} q^{4x_1^2 + 2x_1 x_2 + 5x_2^2 + 4x_1 + x_2}.
\]

Using the ECS transformation
\[
\begin{pmatrix}
0 & -5 \\
1 & 1
\end{pmatrix} \mathbb{Z}^2 +
\begin{pmatrix}
i \\
0
\end{pmatrix}, \quad i = -2, -1, 0, 1, 2,
\]
we obtain
\begin{align} \label{id171}
S = 2f(-q^4, -q^6)f(-q^{76}, -q^{114}) + 2q^8 f(-q^2, -q^8)f(-q^{38}, -q^{152}).
\end{align}

Next, apply another ECS transformation
\[
\begin{pmatrix}
1 & -1 \\
0 & 4
\end{pmatrix} \mathbb{Z}^2 +
\begin{pmatrix}
0 \\
i
\end{pmatrix}, \quad i = -1, 0, 1, 2,
\]
which yields
\begin{align} \label{id172}
S = 2\psi(q^8)\varphi(q^{76}) + 2q^{18} \varphi(q^4)\psi(q^{152}) - 2q^4 \psi(q^2)\psi(q^{38}).
\end{align}

From the 2-dissection of \( \varphi(q) \), letting \( a = b = q \) in \eqref{k=2}, we have
\begin{align} \label{disphi}
\varphi(q) = \varphi(q^4) + 2q \psi(q^8).
\end{align}

Therefore,
\begin{align} \label{id173}
&\varphi(q)\varphi(q^{19}) - \varphi(-q)\varphi(-q^{19}) \notag \\
=&\, \left\{ \varphi(q^4) + 2q \psi(q^8) \right\} \left\{ \varphi(q^{76}) + 2q^{19} \psi(q^{152}) \right\} \notag \\
&\quad - \left\{ \varphi(q^4) - 2q \psi(q^8) \right\} \left\{ \varphi(q^{76}) - 2q^{19} \psi(q^{152}) \right\} \notag \\
=&\, 4q \psi(q^8)\varphi(q^{76}) + 4q^{19} \varphi(q^4)\psi(q^{152}).
\end{align}

Combining \eqref{id171}, \eqref{id172}, and \eqref{id173}, we conclude the proof of \eqref{id17}.
\end{proof}

\begin{rem}
As in the proof of Identity 18, we can show that
\[
\sum_{x_1, x_2 \in \mathbb{Z}} (-1)^{x_2} q^{2x_1^2 + 2x_1 x_2 + 10x_2^2} = \varphi(q^2)\varphi(q^{38}) - 4q^{10} \psi(q^4)\psi(q^{76}),
\]
and
\[
\sum_{x_1, x_2 \in \mathbb{Z}} (-1)^{x_1} q^{2x_1^2 + 2x_1 x_2 + 10x_2^2} = \varphi(-q^2)\varphi(-q^{38}).
\]
If we replace \( q \) by \( q^2 \) in \eqref{id17}, then the right-hand side becomes
\[
\sum_{x_1, x_2 \in \mathbb{Z}} (-1)^{x_2} q^{2x_1^2 + 2x_1 x_2 + 10x_2^2}
- \sum_{x_1, x_2 \in \mathbb{Z}} (-1)^{x_1} q^{2x_1^2 + 2x_1 x_2 + 10x_2^2}.
\]
To justify this identity, one would need to find an infinite sum involving the quadratic form \( (4, 2, 5) \) whose ECS expansion yields this difference. We omit this proof here.
\end{rem}

Identity 19-2 is the last identity of this type, and we sketch its proof here.

\begin{proof}
We can rewrite Identity 19-2 \eqref{ID40192} as
\begin{align} \label{ID192eq}
&2q f(-q^2, -q^3)f(-q^{78}, -q^{114}) + 2q^9 f(-q, -q^4)f(-q^{39}, -q^{156}) \notag\\
=&\, \varphi(-q^3) \varphi(-q^{13}) - \varphi(-q) \varphi(-q^{39}).
\end{align}

Let
\begin{align} \label{aq1920}
A(q) := \sum_{x_1, x_2 \in \mathbb{Z}} (-1)^{x_1} q^{6x_1^2 + 6x_1 x_2 + 8x_2^2},
\end{align}
and apply the ECS associated with the matrix
\[
B = \begin{pmatrix} 1 & -1 \\ 0 & 2 \end{pmatrix}
\]
to obtain
\begin{align}
A(q) = \varphi(-q^6)\varphi(-q^{26}).
\end{align}

Similarly, define
\begin{align} \label{bq1920}
B(q) := \sum_{x_1, x_2 \in \mathbb{Z}} (-1)^{x_1} q^{2x_1^2 + 2x_1 x_2 + 20x_2^2}.
\end{align}
Using the same matrix \( B \), we have
\begin{align} \label{bq192}
B(q) = \varphi(-q^2)\varphi(-q^{78}).
\end{align}
Thus, we can express the right-hand side of \eqref{ID192eq} as \( A(q) - B(q) \), provided we replace \( q \) by \( q^2 \).

Applying the matrix transformation
\[
\left(
\begin{array}{c}
x_1 \\
x_2
\end{array}
\right)
=
\begin{pmatrix}
0 & -2 \\
1 & 1
\end{pmatrix}
\left(
\begin{array}{c}
y_1 \\
y_2
\end{array}
\right)
+
\begin{pmatrix}
i \\
0
\end{pmatrix}, \quad i = 0, 1,
\]
in \eqref{aq1920} yields
\begin{align} \label{192aq1}
A(q) &= \sum_{y_1, y_2 \in \mathbb{Z}} q^{8y_1^2 + 4y_1 y_2 + 20y_2^2}
- \sum_{y_1, y_2 \in \mathbb{Z}} q^{8y_1^2 + 4y_1 y_2 + 20y_2^2 + 6y_1 - 18y_2 + 6}.
\end{align}

By considering the parity of \( x_1 \) in \eqref{bq1920}, we have
\begin{align} \label{192bq1}
B(q) &= \sum_{x_1, x_2 \in \mathbb{Z}} q^{8x_1^2 + 4x_1 x_2 + 20x_2^2}
- \sum_{x_1, x_2 \in \mathbb{Z}} q^{8x_1^2 + 4x_1 x_2 + 20x_2^2 + 8x_1 + 2x_2 + 2}.
\end{align}

Changing variables in \eqref{192aq1} from \( Y \) to \( X \) and subtracting \eqref{192bq1}, we obtain
\begin{align} \label{aq1-bq1}
A(q) - B(q) &= \sum_{x_1, x_2 \in \mathbb{Z}} q^{8x_1^2 + 4x_1 x_2 + 20x_2^2 + 8x_1 + 2x_2 + 2} \notag \\
&\quad - \sum_{x_1, x_2 \in \mathbb{Z}} q^{8x_1^2 + 4x_1 x_2 + 20x_2^2 + 6x_1 - 18x_2 + 6}.
\end{align}

On the other hand, we already have the form of the left-hand side of \eqref{ID192eq} from the proof of Entry 19-1. Let
\[
C(q) := \sum_{x_1, x_2 \in \mathbb{Z}} (-1)^{x_1} q^{5x_1^2 + 2x_1 x_2 + 8x_2^2 + x_1 + 8x_2 + 2}.
\]
We can show that
\[
C(q) = 2q^2 f(-q^4, -q^6)f(-q^{156}, -q^{234}) + 2q^{18} f(-q^2, -q^8)f(-q^{78}, -q^{312}).
\]

Now apply the matrix transformation
\[
\left(
\begin{array}{c}
x_1 \\
x_2
\end{array}
\right)
=
\begin{pmatrix}
0 & 2 \\
1 & 0
\end{pmatrix}
\left(
\begin{array}{c}
y_1 \\
y_2
\end{array}
\right)
+
\begin{pmatrix}
i \\
0
\end{pmatrix}, \quad i = -1, 0,
\]
to \( C(q) \). This gives
\begin{align} \label{cq1}
C(q) &= \sum_{y_1, y_2 \in \mathbb{Z}} q^{8y_1^2 + 4y_1 y_2 + 20y_2^2 + 8y_1 + 2y_2 + 2} - \sum_{y_1, y_2 \in \mathbb{Z}} q^{8y_1^2 + 4y_1 y_2 + 20y_2^2 + 6y_1 - 18y_2 + 6}.
\end{align}

By comparing \eqref{aq1-bq1} and \eqref{cq1}, we conclude that
\[
C(q) = A(q) - B(q).
\]
\end{proof}

This identity links three quadratic forms, each with determinant 39: \( (2, 2, 20) \), \( (6, 6, 8) \), and \( (5, 2, 8) \).

Identities 22, 23, and 24 are of the same type. We provide the proof of Identity 22 here.

\begin{proof}
We multiply both sides of Identity 22 \eqref{ID4022} by \( (-q; -q)_\infty (-q^4; -q^4)_\infty \), so that it becomes equivalent to
\begin{align*}
f(-q^2, q^3)f(-q^8, q^{12}) + q f(-q, q^4)f(-q^4, q^{16}) = f(q, q^3)f(-q^8, -q^8).
\end{align*}

Using \eqref{k=2}, \( f(a,b) = f(a^3b, ab^3) + a f(b/a, a^5b^3) \), we expand each of the four theta functions on the left-hand side. Thus, we need to prove
\begin{align} \label{322quiv}
f(q, q^3)f(-q^8, -q^8) &= \left\{ f(-q^9, -q^{11}) - q^2 f(-q, -q^{19}) \right\} \left\{ f(-q^{36}, -q^{44}) - q^8 f(-q^4, -q^{76}) \right\} \notag \\
&\quad + q \left\{ f(-q^7, -q^{13}) + q f(-q^3, -q^{17}) \right\} \left\{ f(-q^{28}, -q^{52}) + q^4 f(-q^{12}, -q^{68}) \right\}.
\end{align}
We need to find integer solutions to the system of equations
\begin{equation*}
\left\{
\begin{array}{rcl}
4b_{11}^2 + 16b_{21}^2 &=& 20, \\
4b_{12}^2 + 16b_{22}^2 &=& 80, \\
4b_{11}b_{12} + 16b_{21}b_{22} &=& 0.
\end{array}
\right.
\end{equation*}
A solution is given by
\[
B = \begin{pmatrix} 1 & -4 \\ 1 & 1 \end{pmatrix}.
\]
Using the ECS induced by this matrix, we obtain
\begin{align} \label{3221}
f(q, q^3)f(-q^8, -q^8) &= f(-q^9, -q^{11})f(-q^{36}, -q^{44}) + q f(-q^7, -q^{13})f(-q^{28}, -q^{52}) \notag \\
&\quad + q^6 f(-q^3, -q^{17})f(-q^{12}, -q^{68}) + q^3 \psi(-q^5)\psi(-q^{15}) \notag \\
&\quad + q^{10} f(-q, -q^{19})f(-q^4, -q^{76}).
\end{align}

We also consider the zero "complementary" term
\[
-q^2 f(q, q^3)f(-1, -q^{16}) = 0,
\]
and expand it using the same ECS to derive
\begin{align} \label{3222}
-q^2 f(q, q^3)f(-1, -q^{16}) &= q^2 f(-q^3, -q^{17})f(-q^{28}, -q^{52}) - q^2 f(-q, -q^{19})f(-q^{36}, -q^{44}) \notag \\
&\quad - q^3 \psi(-q^5)\psi(-q^{15}) + q^5 f(-q^7, -q^{13})f(-q^{12}, -q^{68}) \notag \\
&\quad - q^8 f(-q^9, -q^{11})f(-q^4, -q^{76}).
\end{align}

Adding \eqref{3221} and \eqref{3222} yields \eqref{322quiv}, as desired.
\end{proof}

For Identity 23 \eqref{ID4023}, we have the equivalent form
\[
f(-q^4, q^6)f(-q^6, q^9) + q f(q^2, -q^8)f(q^3, -q^{12}) = \psi(q)\varphi(-q^{12}).
\]
As in Identity 22, the left-hand side requires an expansion analogous to \eqref{322quiv}. The transformation matrix in this case is
\[
B = \begin{pmatrix} 2 & -3 \\ 1 & 1 \end{pmatrix}.
\]

For Identity 24 \eqref{ID4024}, we write it in the equivalent form
\[
f(q, -q^4)f(-q^{12}, q^{18}) - q f(-q^2, q^3)f(q^6, -q^{24}) = \psi(q^3)\varphi(-q^4).
\]
Here, the transformation matrix is
\[
B = \begin{pmatrix} 1 & -3 \\ 1 & 2 \end{pmatrix}.
\]

\section{Ternary Quadratic Forms and Product Identities for Three Theta Functions}
The idea used in the proof of the identities in Section 3 can be applied to quadratic forms with more than two variables, particularly ternary quadratic forms. In this section, we focus on the case where the matrix \( A \), representing the quadratic form, is not diagonal. Here, we highlight a few noteworthy examples.
\begin{cor}
\begin{align}
\sum_{x_1, x_2, x_3 \in \mathbb{Z}} q^{x_1^2 + x_2^2 + 2x_3^2 + x_1x_3 + x_2x_3}
&= \varphi^2(q) \varphi(q^6) + 8q^2 \psi^2(q^2) \psi(q^{12}), \label{3s1} \\
\sum_{x_1, x_2, x_3 \in \mathbb{Z}} (-1)^{x_1} q^{x_1^2 + x_2^2 + 2x_3^2 + x_1x_3 + x_2x_3}
&= \varphi(q)\varphi(-q)\varphi(-q^6). \label{3s2}
\end{align}
\end{cor}

\begin{proof}
Let
\[
S = \sum_{x_1, x_2, x_3 \in \mathbb{Z}} q^{x_1^2 + x_2^2 + 2x_3^2 + x_1x_3 + x_2x_3}.
\]
It is easy to verify that a matrix \( B \) of the form
\[
B = \begin{pmatrix}
m & 0 & -1 \\
0 & m & -1 \\
0 & 0 & 2
\end{pmatrix}
\]
always diagonalizes \( A \) via the transformation \( B^{\mathrm{T}} A B \). In the simplest case where \( m = 1 \), we have
\[
A = \frac{1}{2} \begin{pmatrix}
2 & 0 & 1 \\
0 & 2 & 1 \\
1 & 1 & 4
\end{pmatrix}, \quad
B^{\mathrm{T}} A B = \begin{pmatrix}
1 & 0 & 0 \\
0 & 1 & 0 \\
0 & 0 & 6
\end{pmatrix}.
\]
Since 
\[
B^* = \begin{pmatrix}
2 & 0 & 1 \\
0 & 2 & 1 \\
0 & 0 & 1
\end{pmatrix},
\]
we find that
\[
\bigcup_{i = 0}^{1} \left\{ B\mathbb{Z}^3 + \begin{pmatrix} 0 \\ 0 \\ i \end{pmatrix} \right\}
\]
is an ECS of \( \mathbb{Z}^3 \). Using this ECS, we expand \( S \) to obtain identity \eqref{3s1}.

We apply the same ECS to
\[
S' = \sum_{x_1, x_2, x_3 \in \mathbb{Z}} (-1)^{x_1} q^{x_1^2 + x_2^2 + 2x_3^2 + x_1x_3 + x_2x_3}.
\]
to obtain identity \eqref{3s2}.
\end{proof}

Next, we choose
\[
\Delta = \begin{pmatrix} 1 \\ 0 \\ 0 \end{pmatrix}, \quad
A = \begin{pmatrix}
2 & 1 & 1 \\
1 & 2 & 0 \\
1 & 0 & 3
\end{pmatrix}
\]
.

Let
\[
B_1 = \begin{pmatrix}
1 & -1 & -2 \\
0 & 2 & 1 \\
0 & 0 & 3
\end{pmatrix}, \quad
B_2 = \begin{pmatrix}
1 & 1 & 2 \\
1 & 0 & -2 \\
-1 & 1 & -1
\end{pmatrix},
\]
then
\[
B_1^{\mathrm{T}} A B_1 =
\begin{pmatrix}
2 & 0 & 0 \\
0 & 6 & 0 \\
0 & 0 & 21
\end{pmatrix}, \quad
B_2^{\mathrm{T}} A B_2 =
\begin{pmatrix}
7 & 0 & 0 \\
0 & 7 & 0 \\
0 & 0 & 7
\end{pmatrix}.
\]

By Theorem~\ref{SIMECS}, we verify that both transformation matrices \( B_1 \) and \( B_2 \) are simple covering matrices. Using the matrix ECS
\[
B_1 \mathbb{Z}^3 + \begin{pmatrix} 0 \\ 0 \\ i \end{pmatrix}, \quad i = -2, \dots, 3,
\]
and
\[
B_2 \mathbb{Z}^3 + \begin{pmatrix} 0 \\ 0 \\ i \end{pmatrix}, \quad i = -3, \dots, 3,
\]
by expanding the infinite sum \[
S = \sum_{x_1, x_2, x_3 \in \mathbb{Z}} (-1)^{x_1} q^{2x_1^2 + 2x_2^2 + 3x_3^2 + 2x_1x_2 + 2x_1x_3},
\]
we obtain the following identity for the product of three theta functions. 

\begin{cor}
\begin{align}
&\varphi(-q^2)\varphi(-q^6)\varphi(q^{21}) - 2q^3 \varphi(-q^2) f(-q^2, -q^{10}) f(q^7, q^{35}) \notag \\
=\, &\varphi(-q^7)^2 \varphi(q^7)
- 2q^2 f(-q^3, -q^{11}) f(-q, -q^{13}) f(q^5, q^9) \notag \\
&+ 2q^2 f(-q^5, -q^9) f(-q, -q^{13}) f(q^3, q^{11})
- 2q^2 f(-q^5, -q^9) f(-q^3, -q^{11}) f(q, q^{13}). \notag
\end{align}
\end{cor}

Next, for
\[
A = \begin{pmatrix}
2 & 0 & 1 \\
0 & 2 & 1 \\
1 & 1 & 2
\end{pmatrix}, \quad
B_1 = \begin{pmatrix}
1 & 1 & 1 \\
1 & -1 & 1 \\
0 & 0 & -2
\end{pmatrix}, \quad
B_2 = \begin{pmatrix}
0 & 2 & 1 \\
1 & 1 & -1 \\
-1 & -1 & -2
\end{pmatrix}, \quad
B_3 = \begin{pmatrix}
1 & 2 & 0 \\
1 & 0 & -2 \\
-2 & 0 & 0
\end{pmatrix}.
\]
we have
\[
B_1^{\mathrm{T}} A B_1 =
\begin{pmatrix}
4 & 0 & 0 \\
0 & 4 & 0 \\
0 & 0 & 4
\end{pmatrix}, \quad
B_2^{\mathrm{T}} A B_2 =
\begin{pmatrix}
2 & 0 & 0 \\
0 & 6 & 0 \\
0 & 0 & 12
\end{pmatrix}, \quad
B_3^{\mathrm{T}} A B_3 =
\begin{pmatrix}
4 & 0 & 0 \\
0 & 8 & 0 \\
0 & 0 & 8
\end{pmatrix}.
\]
We omit the corresponding identities in this paper.
\section{Applications and Summary}
\subsection {Other Applications and Further Discussions}\

We can generate a collection of identities by expanding infinite sums in the form of \eqref{QFGF} as the determinant of the associated quadratic form increases. Since it is impossible to exhaust all such identities, we present here a few particularly elegant analogues of Ramanujan’s forty identities.

Motivated by the proof of \eqref{40ID4} in Section~3, it is natural to seek the next simplest analogue, where the pattern involves one side expressed as the sum of three parts and the other as the sum of seven parts. Such a relation takes the form
\[
2A + 0 = 2B + 2C + 2D + 0,
\]
which simplifies to
\[
A = B + C + D.
\]

In this case, since the determinant of one of the transformation matrices is 7, the identities involve the functions \( A(q), B(q) \), and \( C(q) \), as defined by the Rogers–Selberg identities
\begin{align*}
A(q) &= \sum_{n=0}^\infty \frac{q^{2n^2}}{(q^2;q^2)_n\, (-q;q)_{2n}} = \frac{f(-q^3, -q^4)}{f(-q^2)}, \\
B(q) &= \sum_{n=0}^\infty \frac{q^{2n^2 + 2n}}{(q^2;q^2)_n\, (-q;q)_{2n}} = \frac{f(-q^2, -q^5)}{f(-q^2)}, \\
C(q) &= \sum_{n=0}^\infty \frac{q^{2n^2 + 2n}}{(q^2;q^2)_n\, (-q;q)_{2n+1}} = \frac{f(-q, -q^6)}{f(-q^2)}.
\end{align*}
These identities correspond to entries 33, 32, and 31, respectively, in Slater’s list \cite{slater}.

We present two identities that follow the aforementioned pattern.

\begin{cor}
\begin{align}
f(-q) f(-q^{5}) &= f(-q^2, -q^5) f(-q^{15}, -q^{20}) - q f(-q, -q^6) f(-q^{10}, -q^{25}) \notag \\
&\quad - q^2 f(-q^3, -q^4) f(-q^{5}, -q^{30}), \label{ID304analogue1} \\
f(-q) f(-q^{17}) &= f(-q, -q^6) f(-q^{51}, -q^{68}) - q^2 f(-q^3, -q^4) f(-q^{34}, -q^{85}) \notag \\
&\quad + q^7 f(-q^2, -q^5) f(-q^{17}, -q^{102}). \label{ID304analogue2}
\end{align}
\end{cor}

These identities can also be expressed in the equivalent forms
\begin{align*}
\chi(-q)\chi(-q^{5}) &= B(q) A(q^{5}) - q C(q) B(q) - q^2 A(q) C(q^{5}), \\
\chi(-q)\chi(-q^{17}) &= C(q) A(q^{17}) - q^2 A(q) B(q^{17}) + q^7 B(q) C(q^{17}).
\end{align*}

These are recorded as Identities (3.14) and (3.15) in Hahn’s work \cite{HAHN2003}.

We sketch the proof of \eqref{ID304analogue1} below.

\begin{proof}
Consider the infinite sum
\[
S = \sum_{x_1, x_2 \in \mathbb{Z}} (-1)^{x_2} q^{2x_1^2 + 2x_1x_2 + 3x_2^2 + 2x_1 + x_2}.
\]
Expanding \( S \) using the matrix
\[
B_1 = \begin{pmatrix} 1 & 2 \\ -1 & 1 \end{pmatrix} \quad \text{or} \quad \begin{pmatrix} 0 & -3 \\ 1 & 1 \end{pmatrix},
\]
we obtain
\[
S = 2f(-q^2, -q^4) f(-q^{10}, -q^{20}).
\]

Alternatively, expanding \( S \) using
\[
B_2 = \begin{pmatrix} 1 & -4 \\ 1 & 3 \end{pmatrix},
\]
yields
\begin{align*}
S &= 2f(-q^4, -q^{10}) f(-q^{30}, -q^{40}) - 2q^2 f(-q^2, -q^{12}) f(-q^{20}, -q^{50}) \\
&\quad - 2q^4 f(-q^6, -q^8) f(-q^{10}, -q^{60}).
\end{align*}
We recover \eqref{ID304analogue1} after simplifications.
\end{proof}
As applications, many identities analogous to Ramanujan’s forty identities can be classified as either type I or type II identities. For instance, this applies to certain modular relations for the Göllnitz–Gordon functions in \cite{huang3, huang2}, the septic analogues of the Rogers–Ramanujan functions in \cite{HAHN2003}, the nonic analogues in \cite{Baruah2008}, and the Rogers–Ramanujan-type functions of order eleven in \cite{adiga}.

In a broad sense, proving an identity amounts to demonstrating that  both sides cover the same set of points with the same multiplicities. Some identities can be explained using different ECSs after simple transformations. For example, the identity
\begin{align} \label{lastexample}
(q^3;q^3)^2_\infty = \psi(q^2)\phi(q^9) - q^2\phi(q)\psi(q^{18})
\end{align}
can be rewritten as
\[
2f^2(-q^3, -q^6) = f(1, q^2)f(q^9, q^9) - q^2 f(q, q)f(1, q^{18}),
\]
which may be obtained by expanding the infinite sum
\[
\sum_{x_1, x_2 \in \mathbb{Z}} (-1)^{x_2} q^{2x_1^2 + 2x_1x_2 + 5x_2^2}
\]
using the ECSs associated with the matrices
\[
B_1 = \begin{pmatrix}
1 & -1 \\
0 & 2
\end{pmatrix}, \quad
B_2 = \begin{pmatrix}
1 & -2 \\
1 & 1
\end{pmatrix}.
\]
In this context, identity \eqref{lastexample} can be treated as a type II identity.

On the other hand, with a simple rearrangement, \eqref{lastexample} can also be rewritten as
\[
f^2(-q^3, -q^6) = \phi(q^9)\left(\psi(q^2) - q^2\psi(q^{18})\right) - q^2 \psi(q^{18}) \left(\phi(q) - \phi(q^9)\right).
\]

Applying the 3-dissection formulas for theta functions given in \eqref{3dis}, which correspond to one-dimensional ECSs, we find that
\[
\psi(q^2) - q^2\psi(q^{18}) = f(q^6, q^{12}), \quad \phi(q) - \phi(q^9) = 2f(q^3, q^{15}).
\]
Thus, identity \eqref{lastexample} is equivalent to
\begin{align} \label{lastexample1}
f^2(-q^3, -q^6) = f(q^9, q^9) f(q^6, q^{12}) - q^2 f(1, q^{18}) f(q^3, q^{15}).
\end{align}

Now, to interpret \eqref{lastexample1} in terms of ECS, we look for a matrix \( B \) satisfying
\begin{equation} \label{eqh1}
B^{\mathrm{T}} \begin{pmatrix} 9 & 0 \\ 0 & 9 \end{pmatrix} B = \begin{pmatrix} 18 & 0 \\ 0 & 18 \end{pmatrix}.
\end{equation}
A solution to this equation is
\[
B = \begin{pmatrix} 1 & 1 \\ -1 & 1 \end{pmatrix},
\]
which is the simplest nontrivial matrix for ECS in \( \mathbb{Z}^2 \). This shows that \eqref{lastexample1} is a special case of the identity generated by the transformation induced by \( B \):

If \( ab = cd \), then
\[
f(a,b)f(c,d) = f(ad, bc)f(ac, bd) + a f\left(\frac{c}{a}, a^2bd\right) f\left(\frac{d}{a}, a^2bc\right),
\]
which appears as Entry 29 in \cite[p.~45]{BBNTBK}.

Therefore, the identity \eqref{lastexample1}, which is equivalent to \eqref{lastexample}, may also be treated as a type I identity.

\subsection{Summary}\

First established in 1853, Schröter’s formula and its generalizations allow, under certain conditions, a product of two theta functions to be expressed as a specific sum of products of two theta functions. These formulas have been instrumental in establishing many of Ramanujan’s modular equations.  The general formula known as the Blecksmith–Brillhart–Gerst Theorem, originally established in \cite{BBG1988}, was applied multiple times by the authors in \cite{BER40ID} in their proofs. As noted in \cite{CAOECS}, such identities are special cases of the main theorem developed therein.

In \cite{watson}, Watson approached the product of two theta functions by considering a change of indices in the double sum over \( m \) and \( n \) of the form
\[
\alpha m + \beta n = 5M + a, \quad \gamma m + \delta n = 5N + b,
\]
where \( \alpha, \beta, \gamma, \delta \in \mathbb{Z} \), and \( a, b \in \mathbb{Z} \) are chosen to decompose the original product into a sum of two products of theta functions. Using this method, Watson proved two of Ramanujan’s forty identities and established two additional results used in subsequent proofs.

This change of indices can be expressed as the linear transformation
\[
A \begin{pmatrix} m \\ n \end{pmatrix} = 5 \begin{pmatrix} M \\ N \end{pmatrix} + \begin{pmatrix} a \\ b \end{pmatrix},
\]
where
\[
A = \begin{pmatrix} \alpha & \beta \\ \gamma & \delta \end{pmatrix}.
\]
Since \( \det A = \pm 5 \), this is equivalent to
\[
\begin{pmatrix} m \\ n \end{pmatrix} = \pm A^* \begin{pmatrix} M \\ N \end{pmatrix} + A^{-1} \begin{pmatrix} a \\ b \end{pmatrix},
\]
where \( A^* \) denotes the adjugate of \( A \). This transformation is structurally equivalent to the ECS induced by \( A^* \). To complete the construction, one must identify values of \( a \) and \( b \) such that the vectors \( A^{-1}\begin{pmatrix} a \\ b \end{pmatrix} \) form a complete set of coset representatives.

A generalization of Rogers’ method, as presented in \cite{BER40ID}, takes the following form: Let \( p \) and \( m \) be odd positive integers with \( p > 1 \), and let \( \alpha, \beta, \lambda \in \mathbb{R} \) such that
\[
\alpha m^2 + \beta = \lambda p.
\]
Then,
\begin{align} \label{rogers40}
&\sum_{v \in S_p}
q^{p\alpha m^2 v^2} 
f\left(-q^{p\alpha + 2p\alpha mv}, -q^{p\alpha - 2p\alpha mv} \right)
q^{p\beta v^2}
f\left(-q^{p\beta + 2p\beta v}, -q^{p\beta - 2p\beta v} \right) \notag \\
=& \sum_{u \in \mathbb{Z}} \sum_{t \in \mathbb{Z}} (-1)^t q^{I(u,t)},
\end{align}
where
\[
I(u,t) = \lambda \left( u + \frac{1}{2} + \alpha m t \right)^2 + \frac{\alpha \beta}{\lambda} t^2,
\quad
S_p = \left\{ \frac{1}{2p}, \frac{3}{2p}, \dots, \frac{2p - 1}{2p} \right\}.
\]
Rogers' strategy is to identify two sets of parameters,  
\[\left\{ \alpha_1, \beta_1, m_1, p_1, \lambda_1 \right\}\,\, \hbox{and} \,\, \left\{ \alpha_2, \beta_2, m_2, p_2, \lambda_2 \right\},\]
that yield the same expression on the right-hand side of \eqref{rogers40}. This method establishes an identity between two sums, each representing a product of two theta functions. For instance, if the two sets of parameters are chosen such that   
\[
\alpha_1\beta_1 = \alpha_2\beta_2, \quad \lambda_1 = \lambda_2, \quad 
\frac{\alpha_1 m_1}{\lambda_1} \pm \frac{\alpha_2 m_2}{\lambda_2} \in \mathbb{Z},
\]
then both parameter sets will satisfy the formula for $I$. Consequently, they produce the same sum on the right-hand side of \eqref{rogers40}. 

The proof of \eqref{rogers40} in \cite{BER40ID} proceeds from left to right: under specific conditions, a sum of \( p \) products of theta functions is expressed as a single infinite sum involving a quadratic form. However, it is often more convenient to proceed in the reverse direction—expanding a single term into a sum of \( p \) terms via an ECS. The two parameter sets correspond to two distinct transformation matrices associated with the same extended quadratic form appearing in the infinite sum. The sum can then be expanded using the ECSs determined by these matrices. To express the result as a product of theta functions, the associated quadratic form must be diagonalized.

Bressoud’s main strategies for proving fifteen of Ramanujan’s identities include an extension of Rogers’ method and the use of ``restricted quadratic forms''. These are quadratic forms where variables are subject to congruence restrictions. They can be transformed into ``unrestricted quadratic forms'' by removing those constraints and extending the domain to \( \mathbb{Z}^n \).

Yesilyurt’s generalization of Rogers’ method in \cite{YESJNT} involves expressions of the form  
\[
R(\varepsilon, \delta, \ell, t, \alpha, \beta, m, p, \lambda, x, y),
\]
where \( x \) and \( y \) are variables, and the remaining nine symbols represent parameters. His framework incorporates the quotient form of the quintuple product identity, which is essential in proving several of the forty identities.

The underlying structure of many product identities involving theta functions can be understood in terms of linear transformations—naturally explained via ECS. The approaches of Watson, Rogers, Bressoud, and Yesilyurt are all elementary and fall within the broader matrix ECS framework. Though Bressoud’s and Yesilyurt’s generalizations involve more parameters and complexity, their strategy—like Rogers’—is to identify two parameter sets that yield the same infinite sum on one side of the identity.  This approach has so far been successful only for products of two theta functions.

Although Ramanujan’s original proofs remain unknown, the author believes that his insights are best explained through ECS, or a formulation equivalent to ECS. This framework offers a simple, elegant, and unified structure for understanding many of Ramanujan’s identities. Furthermore, it provides a systematic method for discovering new identities and exploring their applications in partitions.

\end{document}